\documentclass[12pt,reqno]{amsart}
\usepackage[usenames,dvipsnames]{xcolor}
\usepackage{amsmath, amsfonts, epsfig, amssymb, graphics}
\usepackage[colorlinks=true,citecolor=cyan]{hyperref}
\usepackage[mathscr]{euscript}
\newcommand{\pref}[1]{{\rm (\ref{#1})}}

\def\area{\mathrm{area}}
\def\dinv{\mathrm{dinv}}
\def\bounce{\mathrm{b}}
\def\PolyominoParrallelogram{\mathbb{P}}
\def\Paths{\mathrm{P}}
\def\InvariantsSL{\mathcal{R}}
\def\Labelledpolyominoes{\mathbb{L}}
\def\LabelledPaths{\mathrm{L}}
\def\N{\mathbb{N}}
\def\S{\mathbb{S}}

\def\C{\mathbb{C}}
\def\Gr{\mathbb{G}}
\def\Plucker{\mathcal{P}}
\def\SL#1{{\rm SL}_#1}

\def\Frob{\mathrm{Frob}}

\makeatletter
\def\revddots{\mathinner{\mkern1mu\raise\p@
\vbox{\kern7\p@\hbox{.}}\mkern2mu
\raise4\p@\hbox{.}\mkern2mu\raise7\p@\hbox{.}\mkern1mu}}
\makeatother

\def\bz{\overline{0}}
\def\bu{\overline{1}}
\def\bd{\overline{2}}
\def\bt{\overline{3}}
\def\AA{\mathbb{A}}

\def\mbf#1{{\mathbf #1}}

\newcommand{\qbinom}[2]{\genfrac{[}{]}{0pt}{}{\,#1\,}{#2}_q}
\newcommand{\qn}[1]{[\,#1\,]_q}
 
\def\auteur#1{{\sc #1}}
\def\titreref#1{{\em #1}}
\def\vol#1{{\bf #1}}
\def\defn#1{{\em #1}}

\def\jrectangle#1#2{\put(-0.1,0.5){\multiput(0,0)(0,1){#2}{\multiput(0,0)(1,0){#1}{\jaune{\linethickness{6mm}\line(1,0){1.2}}}}}}
\def\vrectangle#1#2{\put(-0.1,0.5){\multiput(0,0)(0,1){#2}{\multiput(0,0)(1,0){#1}{\vertpale{\linethickness{6mm}\line(1,0){1.2}}}}}}

\def\est{\line(1,0){1.05}}
\def\nord{\line(0,1){1.05}}

\def\grille#1#2{\linethickness{.2mm}
   \multiput(.5,1)(0,1){#2}{\gris{\line(1,0){#1}}}
   \multiput(1,.5)(1,0){#1}{\gris{\line(0,1){#2}}}}
\def\coord#1#2#3#4{\linethickness{.3mm}
    \put(0,#2){\vector(1,0){#3}}
    \put(#1,0){\vector(0,1){#4}}}

\def\palecarre{\jaune{\linethickness{\unitlength}\line(1,0){1.1}}}

\def\bleu{\textcolor{blue}}
\def\gris#1{{\color{Gray}#1}}
\def\rouge{\textcolor{red}}
\def\jaune{\textcolor{yellow}}

\def\vertpale#1{{\color{SpringGreen} #1}}

\def\bleux{}
\def\rougex{}

\newtheorem{proposition}{\bleu{Proposition}}

\theoremstyle{definition}
\newtheorem{remark}{\bleu{Remark}}

\begin{document} 

\title[Labelled polyominoes]{\bleu{Combinatorics of  Labelled Parallelogram polyominoes}}
\author[J.-C.~Aval]{J.-C.~Aval}
\address{Labri, CNRS, Universit\'e de Bordeaux,
251 Cours de la Lib√©ration, Bordeaux.}
\author[F.~Bergeron]{F. Bergeron}
\address{D\'epartement de Math\'ematiques, UQAM,  C.P. 8888, Succ. Centre-Ville, 
 Montr\'eal,  H3C 3P8, Canada.}
 \date{November 2012. This work was supported by NSERC-Canada, NSF-USA and CNRS-France}
 \author[A.~Garsia]{A.~Garsia}
\address{Department of Mathematics, UCSD.}
\maketitle
\begin{abstract}
We obtain explicit formulas for the enumeration of labelled parallelogram polyominoes. These are the polyominoes that are bounded, above and below,
by north-east lattice paths going from the origin to a point $(k,n)$. The numbers from $1$ and $n$ (the labels) are bijectively attached to the $n$ north steps of the above-bounding path,
with the condition that they appear in increasing values along consecutive north steps. We  calculate the Frobenius characteristic of the action of the symmetric group $\S_n$ on these labels.  All these enumeration results are  refined to take into account the area of these polyominoes. We make a connection between our enumeration results and the theory of operators for which the integral Macdonald polynomials are joint eigenfunctions. We also explain how these same polyominoes can be used to explicitly construct a linear basis of a ring of $SL_2$-invariants.
\end{abstract}
 \parskip=0pt

{ \setcounter{tocdepth}{1}\parskip=0pt\footnotesize \tableofcontents}
\parskip=8pt

\section{Introduction}
Parallelogram polyominoes have been studied by many authors (see~\cite{bousquet, gessel, viennot} for a nice survey and enumeration results).  
They correspond to pairs $\pi=(\alpha,\beta)$ of north-east paths going from the origin to a point $(k,n)$ in the combinatorial plane $\N\times \N$, with the path $\alpha$ staying ``above'' the path $\beta$. 
Our aim here is to study properties, and related nice formulas, of  ``labelled parallelogram polyominoes''. 
These are obtained by bijectively  labelling each of the $n$ north steps of the path $\alpha$ with the numbers between $1$ and $n$.   
Our motivation stems from a similarity between this new notion and recent work on labelled intervals in the Tamari lattice, in connection with the study of trivariate diagonal harmonic polynomials for the symmetric group (see~\cite{trivariate}). 

We  calculate explicitly the Frobenius characteristic of the natural action of the symmetric group $\S_n$ on these labelled polyominoes; and study aspects of a weighted version of this Frobenius characteristic with respect to the area of the polyominoes. This connects our study to interesting operators for which adequately normalized Macdonald polynomials are joint eigenfunctions. This is the same theory that appears in the study of the $\S_n$-module of bi-variate diagonal harmonics (see~\cite{HHLRU, HMZ, remarkable}).

We also extend some of our considerations to``doubly'' labelled parallelogram polyominoes, with added labels on east steps of the below-bounding path; with a corresponding action of the group $\S_k\times \S_n$. Several components of these spaces are naturally related to parking function modules.

\section{Paths and labelled paths}\label{paths}
A $k\times n$ north-east (lattice) path in $\N\times \N$ is a sequence $\alpha=(p_0,\ldots,p_i,\ldots p_N)$ points $p_i=(x_i,y_i)$  in $\N\times \N$, with $p_0=(0,0)$, such that
    $$\bleux{(x_{i+1},y_{i+1})}=\begin{cases}
      \bleux{(x_i,y_i)+\rougex{(1,0)}} & \text{an \defn{east step}, or } \\[4pt]
      \bleux{(x_i,y_i)+\rougex{(0,1)}} & \text{a \defn{north step}}.
\end{cases}$$
We denote by $\bleux{\Paths_{k,n}}$ the set of north-east paths from 
$(0,0)$ to $(k,n)$ ({\sl i.e.} $p_N=(k,n)$, and hence $N=k+n$). As is well known, these paths number $\binom{n+k}{k}$. They are often encoded as word $\omega=w_1w2\cdots w_N$, on the alphabet $\{x,y\}$, with $x$ standing for an east step, and $y$ for a north step; and $k$ is the number of $x$'s, while $n$ is that of $y$'s. Thus, the path of Figure~\ref{fig0} is encoded as $yyxxxyyxxyxxxyxx$.

Another description of such a path $\alpha$ may be given in terms of the sequence of ``{heights}'' of its $k$ horizontal steps (written from left to right), so that we may write
    $$\bleux{\alpha=a_1a_2\cdots a_k}.$$ 
More specifically, $a_i$ is equal to the unique $y_j$ such that $x_j=i$ and  $x_{j-1}=i-1$. We say that this the \defn{height sequence} description of $\alpha$.  Thus, the top path in Figure~\ref{fig0} may be encoded as the height sequence
$2224455566$. Perforce, the height sequence is increasing, {\sl i.e.} $a_i\leq a_{i+1}$; and any increasing sequence 
$a_1a_2\cdots a_k$, with $0\leq a_i\leq n$, corresponds to a unique path. 

Exchanging the role of the axes, we may also describe $\alpha$ in terms of its \defn{indentation sequence}. This is simply the bottom to top sequence of distance between vertical steps in $\alpha$, and the $y$-axis. For example, the indentation sequence of our running example is $003358$. Once again, this establishes a bijection between $k\times n$ north-east lattice paths, and length $n$ increasing sequence of values between $0$ and $k$.

There is a classical bijection between north-east paths in $\Paths_{k,n}$ and monomials of degree $k$ in the variables $\mbf{x}=x_1,\ldots,x_{n+1}$. One simply associates to a height sequence $\alpha=(a_1,a_2,\ldots,a_k)$, the monomial 
  $$\bleux{\mbf{x}_\alpha:=x_{a_1+1}x_{a_2+1}\cdots x_{a_k+1} }.$$
Thus, we get the well known formula for the \defn{Hilbert series}, denoted $R(x)$, of the polynomial ring $R=\C[\mbf{x}]$:
\begin{eqnarray}
   \bleu{R(x)} &=& \bleu{\sum_{k=0}^\infty \binom{n+k}{k} x^k}\\
      &=& \bleu{\left(\frac{1}{1-x}\right)^{n+1}}.
\end{eqnarray}
Recall that the coefficients of this series correspond to the dimension of the homogeneous components of the space considered.

The $q$-weighted \defn{area} enumeration of $\Paths_{k,n}$, is defined to be
\begin{equation}
   \bleux{\Paths_{k,n}(q):=\sum_{\alpha\in \Paths_{k,n}} q^{\area(\alpha)}},
\end{equation}
with $\area(\alpha)$ equal to the number of cells in $\N\times \N$ lying below the path $\alpha$. In Figure~\ref{fig0}, the area of the path is the shaded region.
It is also well known that  $\Paths_{k,n}(q)$ is equal to the classical $q$-analog of the binomial coefficient:
\begin{eqnarray}
   \bleu{\Paths_{k,n}(q)}&=&\bleu{\qbinom{n+k}{k}}\\
        &=&\bleu{h_{k}(1,q,\ldots,q^{n+1}}).
\end{eqnarray}

\begin{figure}[ht]
\setlength{\unitlength}{5mm}
\def\palecarre{\jaune{\linethickness{\unitlength}\line(1,0){1}}}
\begin{center}
\begin{picture}(11,8)(0,-1)
\put(0,.5){\multiput(0,0)(1,0){10}{\palecarre}
               \multiput(0,1)(1,0){10}{\palecarre}
              \multiput(3,2)(1,0){7}{\palecarre}
               \multiput(3,3)(1,0){7}{\palecarre}
               \multiput(5,4)(1,0){5}{\palecarre}
               \multiput(8,5)(1,0){2}{\palecarre}}
\put(-1,-1){\grille{11}{7}}
\coord{0}{0}{11}{7}
  \linethickness{.5mm}
\put(0,0){\rouge{\line(0,1){2}}}
\put(0,2){\rouge{\line(1,0){3}}}
\put(3,2){\rouge{\line(0,1){2}}}
\put(3,4){\rouge{\line(1,0){2}}}
\put(5,4){\rouge{\line(0,1){1}}}
\put(5,5){\rouge{\line(1,0){3}}}
\put(8,5){\rouge{\line(0,1){1}}}
\put(8,6){\rouge{\line(1,0){2}}}
\put(-2,2.8){$n\left\{\rule{0cm}{45pt}\right.$}
\put(0,-.5){$\underbrace{\hskip140pt}_{\displaystyle k}$}
\end{picture}\end{center}
\caption{A labelled path, and its area.}\label{fig0}
\end{figure}
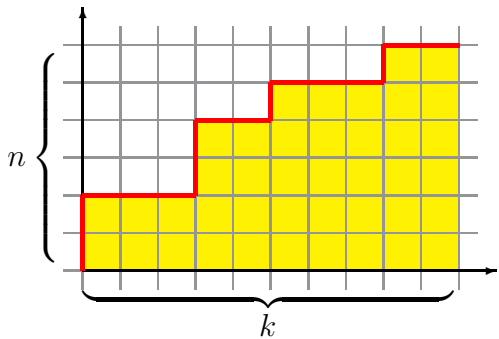

\subsection*{Labelled paths}
We now add labels to north steps of paths in $\Paths_{k,n}$.
We require these labels, going from $1$ to $n$, to be increasing whenever they lie on north steps having same horizontal coordinate. 
The underlying path, of a labelled path $\ell$, is said to be its \defn{shape}. 
We denote by $\LabelledPaths_{k,n}$, the set of labelled paths having shapes lying in $\Paths_{k,n}$. 
It is  easy to see that the number of labelled paths is $(k+1)^n$, with $q$-weighted enumeration given by the straightforward $q$-analog 
\begin{equation}
    \bleu{\sum_{\alpha\in \LabelledPaths_{k,n}} q^{\area(\alpha)}=([k+1]_q)^n}.
\end{equation}
Observe\footnote{Our reason for considering functions from $[n]$ to $[k+1]$ as labelled paths will become apparent in the sequel.} that labelled path $\ell$ may be bijectively turned into a function from $[n]:=\{1,2,\ldots,n\}$ to $[k+1]:=\{1,2,\ldots, k+1\}$, as follows.
We simply set $\ell(i):=j$, whenever $i$ lies along a north step having first coordinate equal to $j-1$. Thus, the fibers of the resulting function correspond to the labels of consecutive
north steps. We say that the \defn{set partition} of $\ell$, is the partition of $[n]$ into these fibers. 
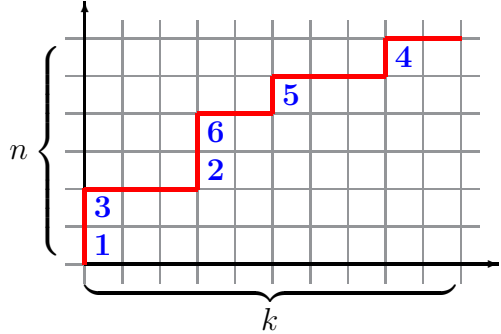
\begin{figure}[ht]
\setlength{\unitlength}{5mm}
\begin{center}
\begin{picture}(11,7)(0,-1)
\put(-1,-1){\grille{11}{7}}
\coord{0}{0}{11}{7}
  \linethickness{.5mm}
\put(0,0){\rouge{\line(0,1){2}}}
\put(0,2){\rouge{\line(1,0){3}}}
\put(3,2){\rouge{\line(0,1){2}}}
\put(3,4){\rouge{\line(1,0){2}}}
\put(5,4){\rouge{\line(0,1){1}}}
\put(5,5){\rouge{\line(1,0){3}}}
\put(8,5){\rouge{\line(0,1){1}}}
\put(8,6){\rouge{\line(1,0){2}}}
\put(-2,2.8){$n\left\{\rule{0cm}{45pt}\right.$}
\put(0,-.5){$\underbrace{\hskip140pt}_{\displaystyle k}$}
\put(0.25,0.25){\put(0,0){\bleu{\bf 1}}
                        \put(0,1){\bleu{\bf 3}}
                        \put(3,2){\bleu{\bf 2}}
                        \put(3,3){\bleu{\bf 6}}
                        \put(5,4){\bleu{\bf 5}}
                        \put(8,5){\bleu{\bf 4}}}
\end{picture}\end{center}
\caption{A labelled path with below area equal to $41$.}\label{figLabelled}
\end{figure}

The symmetric group $\S_n$ acts by permutation of labels, up to reordering labels that lie one above the other. 
In fact, it may be best to consider this in the context of the indentation sequence encoding of paths. In this point of view, a labelled path is simply
a permutation of this sequence. Thus, $k\times n$ labelled paths are just another name for length $n$ sequences of numbers between $0$ and $k$.
The shape  $\alpha$ (underlying path), of a labelled path $\ell$, is simply the increasing ordering of this sequence. Still we will keep using the path terminology, which is better adapted to
our upcoming study of parallelogram polyominoes.

It is classical that the corresponding Frobenius characteristic is given by the formula\footnote{We follow Macdonald's notation (see~\cite{macdonald}) for symmetric functions. We also use plethystic notations. See~\cite{livre} for more on this.}
\begin{eqnarray}
  \bleu{ {\rm Frob}(\LabelledPaths_{k,n})(\mbf{z})}&=&\bleu{ h_n[(k+1)\,\mbf{z}]}\label{chemins}\\
          &=&\bleu{ \sum_{\mu\vdash n} (k+1)^{\ell(\mu)} \frac{p_\mu(\mbf{z})}{z_\mu}}\label{eq:action}\\
         &=&\bleu{\sum_{\mu\vdash n} s_\mu(\underbrace{1,1,\ldots,1}_{k+1})\,s_\mu(\mbf{z})}.
\end{eqnarray}
Recall that this last equality says that 
    $$\bleu{s_\mu(k+1):=s_\mu(\underbrace{1,1,\ldots,1}_{k+1})}$$
is the multiplicity, in $\LabelledPaths_{k,n}$, of the irreducible representation indexed by the partition $\mu$. 
Taking into account the area, we have a direct $q$-analog of the previous formula
\begin{eqnarray}
   \bleu{\Frob(\LabelledPaths_{k,n})(\mbf{z};q)}& =&\bleu{h_n\!\left[\mbf{z}\,\frac{q^{k+1}-1}{q-1}\right] }\label{frobq}\\
      & =&\bleu{\sum_{\mu\vdash n}\frac{1}{z_\mu} \prod_{i\in \mu} \frac{\qn{i\,(k+1)}}{\qn{i}}\, p_i(\mbf{z})}\\
      &=& \bleu{\sum_{\mu\vdash n} s_\mu(1,q,\ldots,q^k)\,s_\mu(\mbf{z})}.
\end{eqnarray}

\section{Parallelogram polyominoes}
Beside setting up notations, the aim of this section is to recall basic (well-known) facts about parallelogram polyominoes.
For two paths $\alpha$ and $\beta$ in $\Paths_{k,n}$, we get a polyomino $\pi=(\alpha,\beta)$ if $\alpha$ stays ``above'' $\beta$.
One usually thinks of a parallelogram polyomino $\pi=(\alpha,\beta)$ as the region of the plane bounded above by the path $\alpha$, and below by the path $\beta$.
Except for endpoints, all points of the path $\beta=(q_0,\ldots,q_i,\ldots q_N)$ are required to be strictly below those of $\alpha=(p_0,\ldots,p_i,\ldots p_N)$. Hence, writing $q_i=(x_i',y_i')$, this is to say that $y'_i<y_j$ whenever $x_i'=x_j$, for $1\leq i,j\leq N$. 
We say that $n$ is the \defn{height} of $\pi$, and that $k$ is its \defn{width}; and we denote by $\bleux{\PolyominoParrallelogram_{k,n}}$ the set of parallelogram polyominoes of height $n$ and width $k$.
The number of $1\times 1$ boxes lying between the two paths is always larger or equal to $k+n-1$, and equality holds for \defn{ribbon shapes}.
Thus, it is natural to subtract this value in the definition of 
the \defn{area} of the polyomino $\pi$, denoted by  $\area(\pi)$. Hence ribbon shapes have zero area.
The \defn{area-enumerating polynomial} is then defined as
   \begin{equation}
      \bleux{\PolyominoParrallelogram_{k,n}(q):=\sum_{\pi\in \PolyominoParrallelogram_{k,n}} q^{\area(\pi)}}.
   \end{equation}
Clearly the reflection in the diagonal $x=y$  maps bijectively the set $\PolyominoParrallelogram_{k,n}$ on the set $\PolyominoParrallelogram_{n,k}$, we have the symmetry \bleux{$\PolyominoParrallelogram_{k,n}(q)=\PolyominoParrallelogram_{n,k}(q)$}.
This is reflected in the \defn{generating function} 
 \begin{eqnarray*}
     \bleux{\PolyominoParrallelogram(x,y;q)}&:=&\bleux{\sum_{(k,n)\in \N\times \N} \PolyominoParrallelogram_{k,n}(q)\,x^k\,y^n},\\ 
          &=&xy+x{y}^{2}+{x}^{2}y+x{y}^{3}+ ( 2+q ) {x}^{2}{y}^{2}+{x}^{3}y\\
         &&\quad\  +x{y}^{4}+ ( 3+2\,q+{q}^{2} ) {x}^{2}{y}^{3}+ ( 3+2\,q+{q}^{2} ) {x}^{3}{y}^{2}+{x}^{4}y\\
&&\quad\ +xy^5+( 4+3\,q+2\,{q}^{2}+{q}^{3} ) {x}^{2}{y}^{4}\\
&&\qquad\qquad + ( 6+6\,q+5\,{q}^{2}+2\,{q}^{3}+{q}^{4} ) {x}^{3}{y}^{3}\\
&&\qquad\qquad + ( 4+3\,q+2\,{q}^{2}+{q}^{3}) {x}^{4}{y}^{2} +x^5y+\ldots\nonumber
   \end{eqnarray*}
 which can be expressed in terms of Bessel functions (see \cite{bousquet}). When $q=1$, we have
\begin{eqnarray}
\bleu{\PolyominoParrallelogram_{k+1,n}(1)}=  \bleu{|\PolyominoParrallelogram_{k+1,n}|} &=&\bleu{s_{k,k}(\underbrace{1,1,\ldots,1}_{n+1})}\label{PPschur}\\
        &=&\bleu{\det\begin{pmatrix}h_k[n+1]  & h_{k+1}[n+1]\\[4pt] h_{k-1}[n+1] & h_{k}[n+1] \end{pmatrix}}\\
         &=&\bleu{\binom{n+k}{k}^2-\binom{n+k+1}{k+1}\binom{n+k-1}{k-1}}\\
         &=&\bleu{\frac{1}{k+1}\binom{n+k}{k}\binom{n+k-1}{k}}.\label{unlabeled}
\end{eqnarray}
Indeed this may be readily deduced using Lindstr\"om-Gessel-Viennot \cite{GV} approach to the enumeration of non-intersecting path configurations in terms of determinants. 
\def\ne{\line(1,1){1}\put(-0.95,0){\line(1,1){1}}}
\def\se{\line(1,-1){1}\put(-0.95,0){\line(1,-1){1}}}
\def\be{\bleu{\line(1,0){1}}}
\def\re{\rouge{\line(1,0){1}}}
\begin{figure}[ht]
\begin{center}
\begin{picture}(11,6)(0,0)
\put(0,.5){\multiput(0,0)(1,0){2}{\palecarre}
               \multiput(0,1)(1,0){6}{\palecarre}
              \multiput(3,2)(1,0){6}{\palecarre}
               \multiput(3,3)(1,0){7}{\palecarre}
               \multiput(5,4)(1,0){5}{\palecarre}
               \multiput(8,5)(1,0){2}{\palecarre}}
\put(-1,-1){\grille{11}{7}}
  \linethickness{.5mm}
\put(0,0){\rouge{\line(0,1){2}}}
\put(0,2){\rouge{\line(1,0){3}}}
\put(3,2){\rouge{\line(0,1){2}}}
\put(3,4){\rouge{\line(1,0){2}}}
\put(5,4){\rouge{\line(0,1){1}}}
\put(5,5){\rouge{\line(1,0){3}}}
\put(8,5){\rouge{\line(0,1){1}}}
\put(8,6){\rouge{\line(1,0){2}}}
\put(0,0){\bleu{\line(1,0){2}}}
\put(2,0){\bleu{\line(0,1){1}}}
\put(2,1){\bleu{\line(1,0){4}}}
\put(6,1){\bleu{\line(0,1){1}}}
\put(6,2){\bleu{\line(1,0){3}}}
\put(9,2){\bleu{\line(0,1){1}}}
\put(9,3){\bleu{\line(1,0){1}}}
\put(10,3){\bleu{\line(0,1){3}}}
\end{picture}
\begin{picture}(17,4)(0,0)
\grille{17}{4}
\put(1,1){\coord{0}{0}{17}{4}}
  \linethickness{.5mm}
\put(1,1){\put(0,0){\ne}  \put(1,1){\ne}  \put(2,2){\se} \put(3,1){\be} \put(4,1){\be} \put(5,1){\ne}  \put(6,2){\ne} 
               \put(7,3){\se} \put(8,2){\be}  \put(9,2){\ne} \put(10,3){\be} \put(11,3){\se} \put(12,2){\be} \put(13,2){\re}
                \put(14,2){\se} \put(15,1){\se}  }
\end{picture}
\end{center}
\caption{A Parallelogram polyomino and the corresponding Motzkin path.}
\label{figPP}
\end{figure}
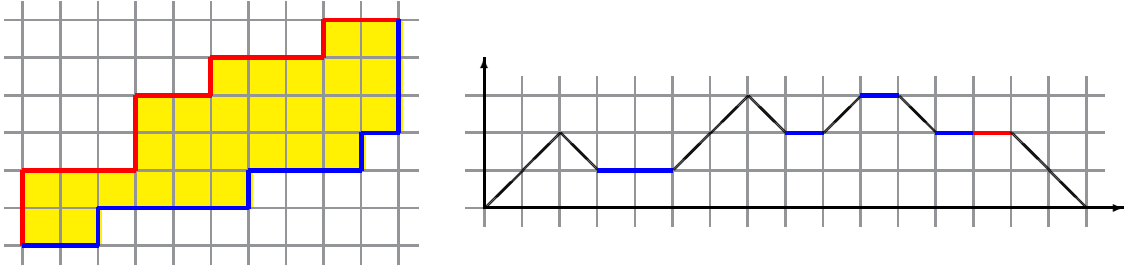

\subsection*{Alternate descriptions} We may shed new light, on our study of parallelogram polyominoes, by exploiting alternate encodings for them. For paths $\alpha,\beta\in \LabelledPaths_{k+1,n}$, respectively\footnote{The shift in indices is intentional here.} given as increasing sequences $\alpha=a_0a_1\cdots a_k$ and $\beta=b_1b_2\cdots b_{k+1}$, the pair $\pi=(\alpha,\beta)$ is a parallelogram polyomino if and only if $a_0=0$, $b_{k+1}=n+1$, and $a_i>b_i$, for $1\leq i\leq k$. Hence, we may identify parallelogram polyominoes with \defn{semi-standard tableaux} (see Figure~\ref{figPPschur}) of shape \bleux{$k^2$}, with values in $\{1,2,\ldots, n+1\}$, thus explaining formula~\pref{PPschur}. Naturally, we may also describe parallelogram polyominoes in terms of indentation sequences,  thus getting semi-standard tableaux of shape $n^2$, with values in $\{1,2,\ldots,n\}$. Going from the first encoding of a given polyomino to the second one, establishes a classical bijection between the two equally numerous families of semi-standard tableaux. 
\def\aa#1{\put(.3,.3){$a_{#1}$}}
\def\bb#1{\put(.3,.3){$b_{#1}$}}
\begin{figure}[ht]
\setlength{\unitlength}{7mm}
\begin{center}
\begin{picture}(8,2)(0,0)
 \linethickness{.3mm}
 \put(0,1){\put(0,0){\aa{1}} \put(1,0){\aa{2}}\put(2,0){\aa{3}} \put(6.5,0){\aa{k}}}
 \put(0,0){\bb{1}} \put(1,0){\bb{2}}\put(2,0){\bb{3}} \put(6.5,0){\bb{k}}
\multiput(0,0)(0,1){3}{\line(1,0){3.5}} 
\multiput(0,0)(1,0){4}{\line(0,1){2}}
\multiput(6,0)(0,1){3}{\line(1,0){1.5}} 
\multiput(6.5,0)(1,0){2}{\line(0,1){2}}
\multiput(4.5,0)(0,1){3}{$\ldots$} 
 \end{picture}\end{center}
\caption{Semi-standard tableau of shape $k^2$; $a_i\leq a_{i+1}$ and $a_i>b_i$.}\label{figPPschur}
\end{figure}

There is also a well-known bijection (see~\cite{viennot}) between parallelogram polyominoes and \defn{Primitive Motzkin Paths}. Recall that these are the paths, in $\N\times \N$, going from $(0,0)$ to $(N,0)$, with steps either north-east, east (either \rouge{red} or \bleu{blue}), or south-east. Thus, the consecutive points, along the path, are linked as follows
    $$\bleux{(x_{i+1},y_{i+1})}=\begin{cases}
      \bleux{(x_i,y_i)+{(1,1)}} & \text{a {north-east step}, or } \\[4pt]
      \bleux{(x_i,y_i)+\rouge{(1,0)}} & \text{a \rouge{red}-{east step}, or } \\[4pt]
     \bleux{(x_i,y_i)+\bleu{(1,0)}} & \text{a \bleu{blue}-{east step}, or } \\[4pt]
     \bleux{(x_i,y_i)+{(1,-1)}} & \text{a {south-east step}}.
\end{cases}$$
Primitive paths are those that never return to the horizontal, except at endpoints. 
Motzkin paths are often simply presented in terms of \defn{Motzkin words}. Recall that these are the words $\omega=w_1w_2\cdots w_N$, on the alphabet $\{d,\overline{d},r,b\}$, such that 
\begin{enumerate}\itemsep=4pt
   \item $\big|\omega_{\leq i}\big|_d\geq \big|\omega_{\leq i}\big|_{\overline{d}}$, for all $1\leq i\leq N-1$,
   \item $\big|\omega\big|_d= \big|\omega\big|_{\overline{d}},$
\end{enumerate}
where $|\omega|_d$ (resp. $|\omega_{\leq i}|_{\overline{d}}$)  stands for the number of $d's$ occurring in $\omega$ (resp. ${\overline{d}}$), and $\omega_{\leq i}$ denotes the prefix of length $i$ of $\omega$. The word corresponds to a primitive path, precisely when all inequalities are strict in the first condition. 
Each letter encodes one of the possible steps: $d$ for north-east, $r$ for red-east, $b$ for blue-east, and $\overline{d}$ for south-east. For a polyomino $\pi=(\alpha,\beta)$, with $\alpha=u_1u_2\cdots u_N$
and $\beta=\alpha=v_1v_2\cdots v_N$ given as words in $\{x,y\}$, Viennot's encoding \cite{viennot} consists in setting
   $$w_i:=\begin{cases}
     d & \text{if}\  (u_i,v_i)=(x,y),\\[4pt]
     r & \text{if}\  (u_i,v_i)=(y,y),\\[4pt]
     b & \text{if}\  (u_i,v_i)=(x,x),\\[4pt]
     \overline{d} & \text{if}\  (u_i,v_i)=(y,x).
\end{cases}$$
See Figure~\ref{figPP} for an example.

Lastly, a parallelogram polyomino may be encoded as a word in the ordered alphabet 
$$\AA=\{\bz,1,\bu,2,\bd,3,\bt,4,\cdots\}.$$
To do this end, we start with its primitive Motzkin path encoding $\pi$,
and successively replace each step in $\pi$ by a $m$-sequence of letters in the $\AA$, $0\leq m\leq 2$, according to the following rules.
\begin{itemize}
\item A down step is replaced by an empty sequence;
\item an up step, at height $j$, is replaced by the $2$-sequence $\overline{j}\,(j+1)$;
\item a red horizontal step, at height $j$, is replaced by $j$;
\item a blue horizontal step at height $j$ is replaced by $\overline{j}$.
\end{itemize}
Figure \ref{fig:Ppiw} illustrates this process.
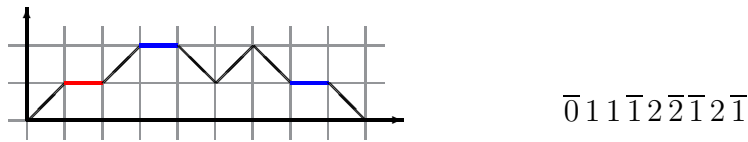
\begin{figure}[ht]
\begin{center}
\begin{picture}(11,5)(2,0)
\grille{10}{3}
\put(1,1){\coord{0}{0}{10}{3}}
  \linethickness{.5mm}
\put(1,1){\put(0,0){\ne}  \put(1,1){\re}  \put(2,1){\ne} \put(3,2){\be} \put(4,2){\se} \put(5,1){\ne}  \put(6,2){\se} 
               \put(7,1){\be} \put(8,1){\se}  }
\end{picture}
\begin{picture}(6,3)(0,0)
\put(2,1){$\bz\,1\,1\,\bu\,2\,\bd\,\bu\,2\,\bu$}
\end{picture}

\end{center}
\caption{From Motzkin path to word in $\AA$.}
\label{fig:Ppiw}
\end{figure}

\begin{remark}
 The correspondence between parallelogram polyominoes and primitive Motz\-kin paths, suggests that we should study ``sequences'' of parallelogram polyominoes which correspond to general Motzkin paths. In particular, this will make apparent more ties with the study of parking functions that correspond to Dyck path touching the diagonal at specific points.
 \end{remark}


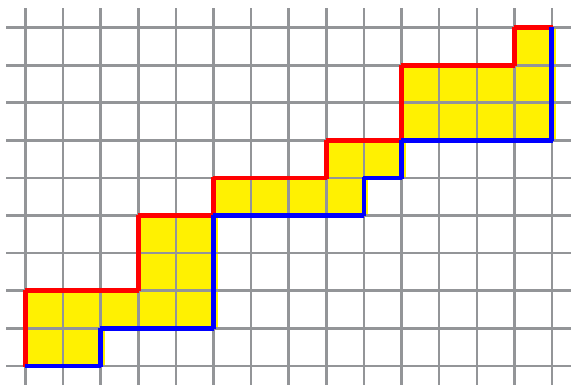
\begin{figure}[ht]
\begin{center}
\begin{picture}(15,10)(0,0)
\put(0,.5){\multiput(0,0)(1,0){2}{\palecarre}
               \multiput(0,1)(1,0){5}{\palecarre}
              \multiput(3,2)(1,0){2}{\palecarre}
               \multiput(3,3)(1,0){2}{\palecarre}
               \multiput(5,4)(1,0){4}{\palecarre}
               \multiput(8,5)(1,0){2}{\palecarre}
               \multiput(10,6)(1,0){4}{\palecarre}
               \multiput(10,7)(1,0){4}{\palecarre}
               \multiput(13,8)(1,0){1}{\palecarre}}
\put(-1,-1){\grille{15}{10}}
  \linethickness{.5mm}
\put(0,0){\rouge{\line(0,1){2}}}
\put(0,2){\rouge{\line(1,0){3}}}
\put(3,2){\rouge{\line(0,1){2}}}
\put(3,4){\rouge{\line(1,0){2}}}
\put(5,4){\rouge{\line(0,1){1}}}
\put(5,5){\rouge{\line(1,0){3}}}
\put(8,5){\rouge{\line(0,1){1}}}
\put(8,6){\rouge{\line(1,0){2}}}
\put(10,6){\rouge{\line(0,1){2}}}
\put(10,8){\rouge{\line(1,0){3}}}
\put(13,8){\rouge{\line(0,1){1}}}
\put(13,9){\rouge{\line(1,0){1}}}
\put(0,0){\bleu{\line(1,0){2}}}
\put(2,0){\bleu{\line(0,1){1}}}
\put(2,1){\bleu{\line(1,0){3}}}
\put(5,1){\bleu{\line(0,1){3}}}
\put(5,4){\bleu{\line(1,0){4}}}
\put(9,4){\bleu{\line(0,1){1}}}
\put(9,5){\bleu{\line(1,0){1}}}
\put(10,5){\bleu{\line(0,1){1}}}
\put(10,6){\bleu{\line(1,0){4}}}
\put(14,6){\bleu{\line(0,1){3}}}
\end{picture}
\end{center}
\caption{A parallelogram polyomino sequence.} \label{figPPS}
\end{figure}

\section{Parallelogram polyominoes, as indexing set of $\SL{2}$-invariants}
One of our motivation, for our current study of the combinatorics of parallelogram polyomino, is to tie this study to that of parking functions and labelled intervals in the Tamari poset.
Recall that more and more evidence shows that this combinatorics is intimately related to the study of the $\S_n$-modules of bivariate and trivariate diagonal harmonics.
Along these lines, it is interesting to observe that the number of polyominoes corresponds to the dimension of the space of $\SL{2}$-invariants of weight $k+1$ of the Grassmanian $\Gr(2,n+1)$, see
 \cite[page 238]{mukai}. We make this more explicit as follows.
Let $X$ be the matrix of $2\times n$ variables
     $$X=\begin{pmatrix} x_1, & x_2,& \cdots &, x_n\\
                                     y_1, & y_2,& \cdots &, y_n \end{pmatrix}.$$
The special group acts on polynomials $f(X)$, in $\C[X]$, by left multiplication of $X$ by matrices in  $\SL{2}$. 
Invariants under this action are precisely the polynomial expressions in the $2\times 2$ minors\footnote{These are clearly $\SL{2}$-invariants.} of $X$.                                 
To construct an explicit linear basis of the resulting ring,  $\InvariantsSL_n:=\C[X]^{\SL{2}}$, one needs only take into account the Pl\"ucker relations on these $2\times 2$ minors.
More precisely, as described in~\cite{miller}, we have the exact sequence
\begin{equation}
    \bleu{0\longrightarrow \Plucker_n \longrightarrow \C[Y_{i,j}]_{1\leq i<j\leq n} \longrightarrow \InvariantsSL_n \longrightarrow 0},
\end{equation}
  where the $Y_{i,j}$ are variables, the middle arrow sends $Y_{i,j}$ to $X_{i,j}$ (the $2\times 2$ minor corresponding to the choice of columns $i$ and $j$ in $X$), and $\Plucker_n$ stands for the ideal generated by the Pl\"ucker relations
     \begin{equation}
        \bleux{Y_{i\ell}Y_{jk}-Y_{ik}Y_{j\ell}+Y_{ij}Y_{k\ell} =0},\qquad {\rm for}\quad\bleux{1\leq i<j<k<\ell\leq n}.
     \end{equation}
 Moreover, using a somewhat reformulated result also discussed in~\cite[Thm 14.6]{miller}, it follows that a linear basis of $\InvariantsSL_n=\C[X]^{\SL{2}}$ may be indexed by parallelogram polyominoes. In fact, we do this in a graded fashion\footnote{Considering $2\times 2$ minors to be of degree $1$.}, constructing for each $d\geq0$ the $d$-homogeneous part of the basis as a collection of product of $d$ minors. The characterizing property of these products is that they must not be divisible by a binomial $X_{ij}X_{k\ell}$, for $\{i,j\}$ incomparable to $\{k,\ell\}$ relative to the following order on pairs of integers. We set
  $$\bleux{\{i,j\} \prec \{k,\ell\}},\qquad {\rm iff}\qquad \bleux{i<k}\ {\rm and}\ \bleux{j<\ell},$$
  assuming, without loss of generality, that $i<j$ and $k<\ell$.
  
Now, given a polyomino $\pi=(\alpha,\beta)\in \PolyominoParrallelogram_{d+1,n-1}$, for which the height sequences of $\alpha$ and $\beta$ are respectively\footnote{The shift in indices is intentional here.} $a_0a_1\cdots a_d$ and $b_1b_2\cdots b_{d+1}$, we consider the minor monomial 
\begin{equation}
  \bleux{X_\pi:=\prod_{i=1}^d X_{(b_i+1,a_i+1)}}.
\end{equation}
The fact that the path $\beta$ remains below the path $\alpha$ is equivalent to the fact that the minor monomial $X_\pi$ satisfies the charactering property of the previous paragraph. For example, the minor monomial associated to the polyomino $\pi=(2224455566,0011112223)$ of Figure~\ref{fig1} is
    $$X_\pi= X_{13}X_{23}^2X_{25}^2X_{36}^3X_{47}.$$
 \begin{proposition}
The family $\{X_\pi\}_\pi$, with $\pi$ varying in the set of parallelogram polyominoes $\PolyominoParrallelogram_{d+1,n-1}$, constitute a linear basis for the degree $d$ homogeneous component of the ring $\InvariantsSL_n$. 
\end{proposition}
 It follows that, for a given $n$, the Hilbert series of $\InvariantsSL_n$ is given   by the formula
     \begin{eqnarray}
       \bleu{\InvariantsSL_n(x)}& =&\bleu{ \sum_{d=0}^\infty |\PolyominoParrallelogram_{k+1,n}|\, x^k}\\
           &=&\bleu{\frac{1}{1-x^{2\,n-1}}  \sum_{k=0}^{n-2} \frac{1}{k+1}\,\binom{n-2}{k}\binom{n-1}{k}\, x^k }.
     \end{eqnarray}
 This approach may readily be expanded to cover families of $r$ non-intersecting paths from $(0,0)$ to $(k,n)$, which also parametrize $\SL{r}$-invariants. The corresponding enumeration is easily obtained in terms of Schur functions (using Lindstr\"om-Gessel-Viennot \cite{GV}), giving
 \begin{equation}
    \bleu{|\PolyominoParrallelogram_{k+1,n}^{(r)}| = s_{k^r}(\underbrace{1,1,\ldots,1}_{n+1})},
 \end{equation}
 where $\PolyominoParrallelogram_{k,n}^{(r)}$ denotes the set of such families, and $k^r$ stands for the rectangular partition $(\underbrace{k,k,\ldots,k}_r)$.

\subsection*{Action of $\S_n$ on $\InvariantsSL_n$}
We may refine the graded enumeration of the ring $\InvariantsSL_n$, by considering its irreducible decomposition with respect to the action of $\S_n$, which corresponds to permutations of the $n$ columns of $X$.
Thus we consider the graded Frobenius characteristic
    \begin{equation}
       \bleux{\Frob(\InvariantsSL_n)(\mbf{z};x)=\sum_{d=0}^\infty \Frob(\InvariantsSL_n^{(d)})(\mbf{z})\, x^d},
    \end{equation}
 where $\InvariantsSL_n^{(d)}$ stands for the degree $d$ homogeneous component of $\InvariantsSL_n$.
 
 To calculate this, we follow an approach due to Littlewood for the calculation of the character of the restriction to $\S_n$ of an action of $GL_n$.
 It follows from general principles that the character of the action of $GL_n$ on $\InvariantsSL_n^{(d+1)}$, the degree $d+1$ component of $\InvariantsSL_n$, is equal to the schur function $s_{dd}(q_1,\ldots,q_n)$. 
 The corresponding character value, for the restriction of this action to $\S_n$, is obtained by evaluating this Schur function at the eigenvalues of permutations. 
 This may be easily calculated, without explicit knowledge of these eigenvalues, in the following manner. 
 
 We start by expanding $s_{dd}(q_1,\ldots,q_n)$ in terms of the power-sum symmetric function, thus turning the required evaluation into the calculation of the sums of $k^{\rm th}$-powers of the eigenvalues of any given permutation matrix $\sigma$. But this turns out to be equal to the number of fixed points of the $k^{\rm th}$-power of $\sigma$. If $\mu$ is the partition giving cycle sizes of $\sigma$, this number of fixed points may simply be expressed as
 \begin{equation}
     \bleux{{\rm fix}(\sigma^k)}=\bleux{\sum_{\mu_i|k}\mu_i}.
 \end{equation}
Hence, the Frobenius characteristic that we are looking for is given by the formula
 \begin{equation}
    \bleu{ \InvariantsSL_n^{(d+1)}(\mbf{z})}=\bleu{\sum_{\mu\vdash n} s_{dd}\big|_{p_k\leftarrow \varphi_k(\mu)} \frac{p_\mu(\mbf{z})}{z_\mu}},
 \end{equation}
 where we set $\varphi_k(\mu):=\sum_{\mu_i|k}\mu_i$, {\sl i.e.} the sum of the parts of $\mu$ that divide $k$.
 

\section{Labelled parallelogram polyominoes} A \defn{labelled parallelogram polyomino} is obtained by labeling (as in Section~\ref{paths}) the top path of a parallelogram polyomino.  
Equivalently, this may be presented in terms of \defn{labelled Motzkin paths}, with labels associated with either north-east steps, or red-east steps. In this case, the relevant increasing-labelling condition simply states that labels should be increasing whenever we they occur on consecutive north-east or red-east steps.
Figure~\ref{figPP} gives an example of such a labelling, in both encodings.
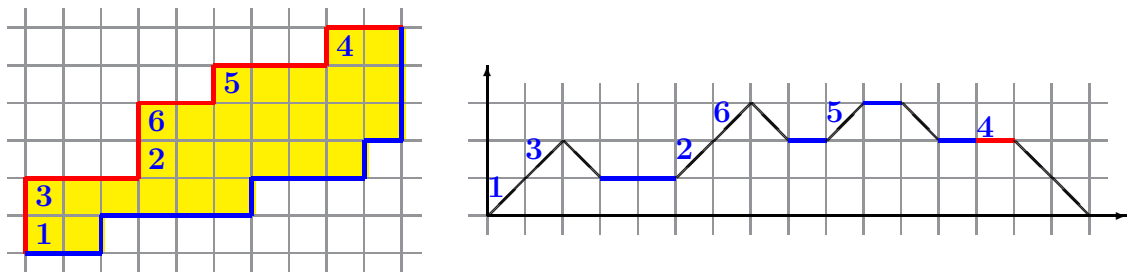
\begin{figure}[ht]
\begin{center}
\begin{picture}(11,7)(0,0)
\put(0,.5){\multiput(0,0)(1,0){2}{\palecarre}
               \multiput(0,1)(1,0){6}{\palecarre}
              \multiput(3,2)(1,0){6}{\palecarre}
               \multiput(3,3)(1,0){7}{\palecarre}
               \multiput(5,4)(1,0){5}{\palecarre}
               \multiput(8,5)(1,0){2}{\palecarre}}
\put(-1,-1){\grille{11}{7}}
  \linethickness{.5mm}
\put(0,0){\rouge{\line(0,1){2}}}
\put(0,2){\rouge{\line(1,0){3}}}
\put(3,2){\rouge{\line(0,1){2}}}
\put(3,4){\rouge{\line(1,0){2}}}
\put(5,4){\rouge{\line(0,1){1}}}
\put(5,5){\rouge{\line(1,0){3}}}
\put(8,5){\rouge{\line(0,1){1}}}
\put(8,6){\rouge{\line(1,0){2}}}
\put(0,0){\bleu{\line(1,0){2}}}
\put(2,0){\bleu{\line(0,1){1}}}
\put(2,1){\bleu{\line(1,0){4}}}
\put(6,1){\bleu{\line(0,1){1}}}
\put(6,2){\bleu{\line(1,0){3}}}
\put(9,2){\bleu{\line(0,1){1}}}
\put(9,3){\bleu{\line(1,0){1}}}
\put(10,3){\bleu{\line(0,1){3}}}
\put(0.25,0.25){\put(0,0){\bleu{\bf 1}}
                        \put(0,1){\bleu{\bf 3}}
                        \put(3,2){\bleu{\bf 2}}
                        \put(3,3){\bleu{\bf 6}}
                        \put(5,4){\bleu{\bf 5}}
                        \put(8,5){\bleu{\bf 4}}}
\end{picture}
\begin{picture}(17,4)(0,0)
\grille{17}{4}
\put(1,1){\coord{0}{0}{17}{4}}
  \linethickness{.5mm}
\put(1,1){\put(0,0){\ne}  \put(1,1){\ne}  \put(2,2){\se} \put(3,1){\be} \put(4,1){\be} \put(5,1){\ne}  \put(6,2){\ne} 
               \put(7,3){\se} \put(8,2){\be}  \put(9,2){\ne} \put(10,3){\be} \put(11,3){\se} \put(12,2){\be} \put(13,2){\re}
                \put(14,2){\se} \put(15,1){\se}  }
\put(1,1.5){\put(0,0){\bleu{\bf 1}}
                        \put(1,1){\bleu{\bf 3}}
                        \put(5,1){\bleu{\bf 2}}
                        \put(6,2){\bleu{\bf 6}}
                        \put(9,2){\bleu{\bf 5}}
                        \put(13,1.6){\bleu{\bf 4}}}
\end{picture}
\end{center}
\caption{A labelled parallelogram polyomino.}\label{fig_polyomino}
\label{fig1}
\end{figure}

We  say that the underlying parallelogram polyomino is the \defn{shape} of the labelled parallelogram polyomino, and denote by $\Labelledpolyominoes_\pi$ the set of labelled parallelogram polyominoes of shape $\pi$.
We also denote by $\Labelledpolyominoes_{k,n}$ the set of labelled parallelogram polyominoes of height $n$ and width $k$, so that
   $$\bleux{\Labelledpolyominoes_{k,n}=\bigcup_{\pi\in \PolyominoParrallelogram_{k,n}}\Labelledpolyominoes_\pi}.$$

A maximal length sequence of north steps, in a path $\alpha$, is said to be a \defn{big-step} of $\alpha$. Clearly, for a height $n$ path $\alpha$, the sequence of big-step lengths in $\alpha$ forms a composition of $n$. We denote it by $\bleux{\gamma(\alpha)}$, also setting $\bleux{\gamma(\pi):=\gamma(\alpha)}$ in the case of a parallelogram polyomino $\pi=(\alpha,\beta)$. Recall that it is customary to write \bleux{$\nu\models n$} when $\nu$ is a \defn{composition of} $n$.

\begin{remark}\label{la_remarque} It is surely worth mentioning that, when $k=n$, ``classical''  Parking Functions may be obtained as special instances of parallelogram polyomino (sequences\footnote{This is needed to take into account the possibility that the two paths may touch one another.}), in fact in at least three different ways. The first is obtained by the simple device of constraining the lower path to be equal to the \defn{zigzag path}, the one who has height sequence equal to $012\cdots (n-1)$. 
In a similar manner, we may fix the lower path to be the one that follows the lower boundary of the surrounding rectangle. Still, an even more natural way is to consider 
labelled parallelogram polyominoes sequences whose shape is symmetric with respect to the diagonal $x=y$. 
All of these correspondences are compatible with the action on labels, and area is easy to track. 
Hence, most of the following formulas may be considered as extensions of analogous formulas for parking functions, potentially in three different ways. 
\end{remark}

For a height $n$ parallelogram polyomino $\pi=(\alpha,\beta)$, the symmetric group $\S_n$ acts by permutation on the set $\Labelledpolyominoes_\pi$, up to a reordering of labels along big-steps. This implies that
 \begin{equation}
   \bleux{|\Labelledpolyominoes_\pi|=\binom{n}{\gamma(\pi)}},
 \end{equation}  
where we use the classical multinomial coefficient notation
   $$\binom{n}{\nu}=\frac{n!}{\nu_1!\nu_2!\cdots \nu_k!},\qquad {\rm when}\qquad \nu=(\nu_1,\nu_2,\ldots,\nu_k).$$
It follows (see Proposition~\ref{prop_frob_pol}) that
\begin{proposition}
The number of labelled parallelogram polyominoes is given by the formula
   \begin{equation}\label{labelled}
      \bleu{|\Labelledpolyominoes_{k,n}| = k^{n-1}\,\binom{n+k-2}{k-1}},
   \end{equation}
 and we have the identity
    \begin{equation}\label{labelled2}
      \bleu{k^{n-1}\,\binom{n+k-2}{k-1} =\sum_{\pi\in \PolyominoParrallelogram_{k,n}} \binom{n}{\gamma(\pi)}}.
   \end{equation} 
   \end{proposition}
 It may be worth mentioning that we have the nice generating function
    \begin{equation}\label{labelled3}
      \bleux{\sum_{n=0}^\infty k^{n-1}\,\binom{n+k-2}{k-1} x^n=\left(\frac{1}{1-k\,x}\right)^k}.
   \end{equation}
The following result provides a common refinement of formulas~\pref{unlabeled} and~\pref{labelled}, stated in terms of the Frobenius characteristic  of the $\S_n$-modules $\Labelledpolyominoes_\pi$ and $\Labelledpolyominoes_{k,n}$ (here considered as free vector spaces over the specified sets).


Writing $\nu=(\nu_1,\nu_2,\ldots,\nu_j)$ for $\gamma(\pi)$, it is classical that the Frobenius  characteristic of $\Labelledpolyominoes_\pi$ is simply given by the symmetric function $h_{\nu}(\mbf{z})$ (in the variables $\mbf{z}=z_1,z_2,z_3,\ldots$). Indeed, the Young subgroup $\S_{\nu_1}\times \cdots\times \S_{\nu_j}$ is  precisely the stabilizer of the labelled parallelograms polyominoes of shape $\pi$. We thus have that 
\begin{proposition}\label{prop_frob_pol}
   The Frobenius characteristic of $\Labelledpolyominoes_{k,n}$ is given by the formula
      \begin{equation}
         \bleux{{\rm Frob}(\Labelledpolyominoes_{k,n}):=\sum_{\pi\in \PolyominoParrallelogram_{k,n}} h_{\gamma(\pi)}(\mbf{z})},
      \end{equation}
and
     \begin{eqnarray}
         \bleu{ {\rm Frob}(\Labelledpolyominoes_{k,n})} 
                        &=&\bleu{ \frac{1}{k}\,h_n[k\,\mbf{z}]\,\binom{n+k-2}{k-1}}\label{eq:Frob}\\
                        &=& \sum_{\lambda\vdash n}\bleux{{k^{\ell(\lambda)-1}} \binom{n+k-2}{k-1}}\, \frac{ \,p_\lambda(\mbf{z})}{z_\lambda} \label{eq:Frob_p}\\
                        &=&\sum_{\mu\vdash n} \bleux{\frac{1}{k} s_\mu(\underbrace{1,1,\ldots,1}_k)\binom{n+k-2}{k-1}}\,s_\mu(\mbf{z}),\label{char_parkn}
      \end{eqnarray}   
      where $\ell(\nu)$ {\rm (}resp. $\ell(\mu)${\rm )} stands for the length of a composition $\nu$ {\rm (}resp. partition $\mu${\rm )}, and $k$ is considered as a constant for the plethystic evaluation in the formulas.\end{proposition}
Observe that  Formulas~\pref{labelled} and~\pref{unlabeled} follow easily from the above expressions, and that~\pref{eq:Frob}--\pref{char_parkn} are polynomial in $k$. Special values are as follows:
\begin{eqnarray*}
   {\rm Frob}(\Labelledpolyominoes_{k,1})&=&h_{{1}}(\mbf{z})\\
   {\rm Frob}(\Labelledpolyominoes_{k,2})&=&{k\choose 2} h_{11}(\mbf{z})+k\,h_{2}(\mbf{z})\\
   {\rm Frob}(\Labelledpolyominoes_{k,3})&=&2\,\binom{k+1}{4} h_{111}(\mbf{z})+3\,\binom{k+1}{3}h_{21}(\mbf{z})+\binom{k+1}{2}h_{3}(\mbf{z})\\
   {\rm Frob}(\Labelledpolyominoes_{k,4})&=&5\,\binom{k+2}{ 6}h_{1111}(\mbf{z})+10\,\binom{k+2}{ 5}h_{211}(\mbf{z})+
4\,\binom{k+2}{ 4}h_{31}(\mbf{z})\\
 &&\qquad\qquad +2\,\binom{k+2}{ 4} h_{22}(\mbf{z})+\binom{k+2}{3}h_{4}(\mbf{z})
\end{eqnarray*}


\begin{proof}[Proof of \pref{eq:Frob}]
An extension of the ``cyclic lemma'' of  \cite{ADDL-cyclic}  gives a combinatorial proof of Identity~\pref{eq:Frob}.
Consider the action of $\S_n$, on pairs $(\ell',\beta')$ in $\LabelledPaths_{k-1,n}\times \Paths_{k-1,n-1}$, that permutes labels in $\ell'$ (as in Section~\ref{paths}).
Assume that $\alpha'$ is the underlying path of $\ell'$.
Since $\ell'$ and $\beta$ are independent, in view of~\pref{chemins},  the corresponding Frobenius may be expressed as
\begin{equation}
\Frob(\LabelledPaths_{k-1,n}\times \Paths_{k-1,n-1})(\mbf{z}) = h_n[k\,\mbf{z}]\,\binom{n+k-2}{k-1}\label{eq:Frob-u}.
\end{equation}
We are going to show that there is an equivalence relation on the set $\LabelledPaths_{k-1,n}\times \Paths_{k-1,n-1}$, whose equivalence classes are all of size $k$, each of which containing
one and only one labelled polyomino. Furthermore, for all pairs of a given class, the set partition associated to the labelled path are all the same. Hence,
we may conclude that $\Frob(\Labelledpolyominoes_{k,n})$ is equal to the right hand-side of~\pref{eq:Frob-u} divided by $k$, as formulated in \pref{eq:Frob}. 

\begin{figure}[ht]
\begin{center}
\begin{picture}(10,3)(0,0)

 \linethickness{1mm}
\put(0,0){
\vrectangle{4}{3}}
\put(-1,-1){\grille{5}{4}}
 \linethickness{.5mm}
\put(0,0.05){
\put(0,0){\rouge{\est}}
\put(1,0){\rouge{\nord}\put(.2,.25){\rouge{\footnotesize$2$}}}
\put(1,1){\rouge{\est}}
\put(2,1){\rouge{\est}}
\put(3,1){\rouge{\nord}\put(.2,.25){\rouge{\footnotesize$1$}}}
\put(3,2){\rouge{\nord}\put(.2,.25){\rouge{\footnotesize$3$}}}
\put(3,3){\rouge{\est}}}
\put(0.05,0){
\put(0,0){\bleu{\est}}
\put(1,0){\bleu{\est}}
\put(2,0){\bleu{\nord}}
\put(2,1){\bleu{\nord}}
\put(2,2){\bleu{\est}}
\put(3,2){\bleu{\est}}
\put(4,2){\bleu{\nord}}}

\put(5.1,1.2){\large{$\equiv$}}
\put(7,0){
\put(0,0){\jrectangle{4}{3}}
\put(-1,-1){\grille{5}{4}}
 \linethickness{.5mm}
\put(1,0){
\put(0,0){\bleu{\nord}}
\put(0,1){\bleu{\est}}
\put(1,1){\bleu{\est}}
\put(2,1){\bleu{\nord}}
\put(2,2){\bleu{\est}}
\put(3,2){\bleu{\nord}}}
\put(0,0){\rouge{\nord} \put(.2,.25){\rouge{\footnotesize$1$}}}
\put(0,1){\rouge{\nord} \put(.2,.25){\rouge{\footnotesize$3$}}}
\put(0,2){\rouge{\est}}
\put(1,2){\rouge{\est}}
\put(2,2){\rouge{\nord} \put(.2,.25){\rouge{\footnotesize$2$}}}
\put(2,3){\rouge{\est}}
\put(3,3){\rouge{\est}}}
\end{picture}\end{center}
\caption{Only polyomino in the equivalence class of a pair $(\ell,\beta)$.}
\label{fig:block}
\end{figure}
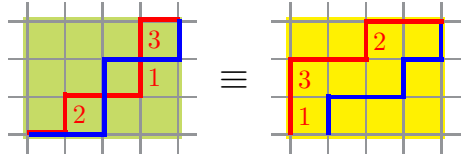

To this end, we work with slightly modified versions, denoted by $\alpha$ and $\beta$, of the paths $\alpha'$ and $\beta'$.
On one hand, $\alpha$ is obtained by simply adding an east step at the end of $\alpha'$. On the other hand, we get $\beta$ by first shifting $\beta'$
to the right (sending each vertex $(a,b)$ to $(a+1,b)$), and then adding a final north step. Hence $\alpha$ now goes from $(0,0)$ to $(k,n)$, while
$\beta$ goes from $(1,0)$ to $(k,n)$. We further denote by $\ell$ the labelled path version of $\ell'$, which is directly obtained by carrying over the labels of $\alpha'$ to $\alpha$. This makes sense since the only difference
between the two is the perforce unlabeled final step. We underline that by construction, $\ell$ and $\beta$ end up at the same point, with a final east step for $\ell$ and a final north step for $\beta$. 

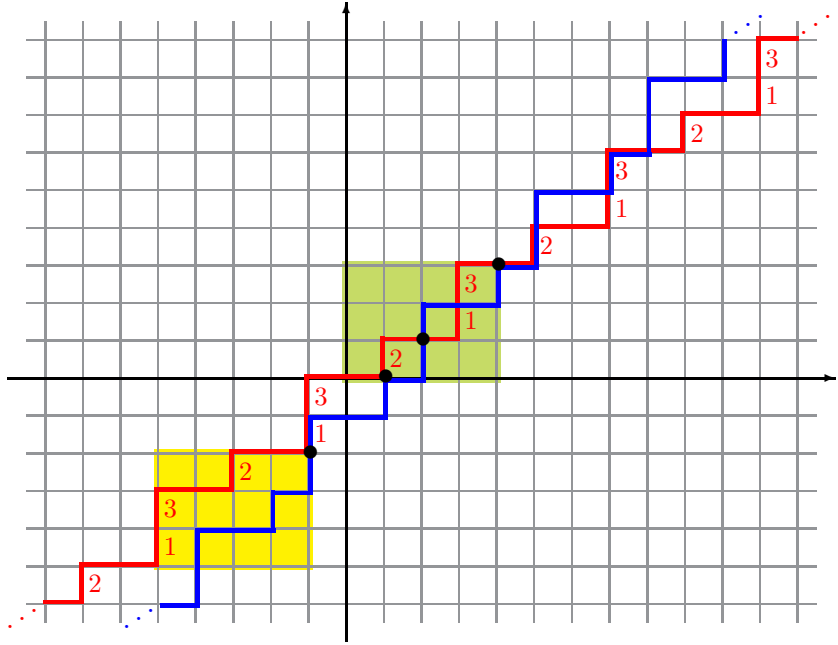
\begin{figure}[ht]
\begin{center}
\begin{picture}(22,16)(0,1)
   \put(9,7){\vrectangle{4}{3}}
   \put(4,2){\jrectangle{4}{3}}
   \grille{21}{16}
   \coord{9}{7}{22}{17}
 \linethickness{.5mm}
\multiput(0.95,1.05)(4,3){5}{
\put(0,0){\rouge{\est}}
\put(1,0){\rouge{\nord}\put(.2,.25){\rouge{\footnotesize$2$}}}
\put(1,1){\rouge{\est}}
\put(2,1){\rouge{\est}}
\put(3,1){\rouge{\nord}\put(.2,.25){\rouge{\footnotesize$1$}}}
\put(3,2){\rouge{\nord}\put(.2,.25){\rouge{\footnotesize$3$}}}
\put(3,3){\rouge{\est}}}
\multiput(4.05,0.95)(3,3){5}
{\put(0,0){\bleu{\est}}
\put(1,0){\bleu{\nord}}
\put(1,1){\bleu{\nord}}
\put(1,2){\bleu{\est}}
\put(2,2){\bleu{\est}}
\put(3,2){\bleu{\nord}}}
\put(0.05,0.05){
\put(8,5){{\circle*{.35}}} 
\put(10,7){{\circle*{.35}}} 
\put(11,8){{\circle*{.35}}} 
\put(13,10){{\circle*{.35}}}}
\put(-0.1,0.3){\rouge{$\revddots$}}
\put(3,0.3){\bleu{$\revddots$}}
\put(21,16.1){\rouge{$\revddots$}}
\put(19.2,16.1){\bleu{$\revddots$}}
\end{picture}\end{center}
\caption{The $k$ transversal intersections of $\infty$-iterates of binomial paths.}
\label{fig:br}
\end{figure}

We now consider the two bi-infinite (labelled in one case) paths, denoted by $\ell^\infty$ and $\beta^\infty$, respective obtained by repeatedly gluing together copies of the paths $\ell$ and $\beta$
at both of their ends. This is illustrated in Figure~\ref{fig:br}. Two pairs $(\ell_1',\beta_1')$ and $(\ell_2',\beta_2')$, in $\LabelledPaths_{k-1,n}\times \Paths_{k-1,n-1}$, are considered to be equivalent if they give rise to the same pair of bi-infinite paths, {\sl i.e.} $(\ell_1^\infty,\beta_1^\infty)$ is equal to $(\ell_2^\infty,\beta_2^\infty)$ up to a shift of the origin. Equivalent pairs correspond to specific portions of the pair of infinite paths.
These portions are precisely characterized by points $(x,y)$ where the two paths intersect in such a way that the $\ell^\infty$ reaches $(x,y)$ by an east step, and $\beta^\infty$ by a north step. We then say that $\ell^\infty$ and $\beta^\infty$ have a \defn{transversal intersection}  at $(x,y)$. Equivalent pairs correspond to portions of $\ell^\infty$ and $\beta^\infty$ both lying in the rectangle having south west corner $(x-k,y-n)$ and north east corner $(x,y)$, for transversal intersection points $p=(x,y)$. We denote by $\ell_p$ and $\beta_p$ these portions.

As shown in~\cite{ADDL-cyclic} the two paths $\ell^\infty$ and $\beta^\infty$ have exactly $k$ transversal intersections.
To finish our proof, we need only observe that, for all transversal intersection point $p$, the set partition of $\ell_p$ is equal to that of $\ell$. Indeed, the big steps of $\ell_p$ are simply a cyclic shift of those of $\ell$, together with their labels.
\end{proof}

\begin{remark}\label{remarque_trois}
It is interesting to compare Formula~\pref{eq:Frob_p}, evaluated at $k=r\,n$, with the conjectured expression \cite[Formula~(17)]{trivariate} 
for the Frobenius characteristic of the space of  trivariate generalized diagonal harmonics. Recall that, up to a sign twist, this expression takes the form
 \begin{equation}\label{formule_frob_p}
      \sum_{\lambda \vdash n} \bleux{(r\, n+1)^{\ell(\lambda)-2}
                      \prod_{j\in\lambda} \binom{(r+1)\,j}{j}}\,\frac{ p_\lambda(\mbf{z})}{z_\lambda}.
\end{equation}
It has been shown in~\cite{preville} that this formula gives the Frobenius characteristic for the action of the symmetric group on labelled intervals in the $r$-Tamari posets. Our point here is that the difference 
 \begin{equation} \sum_{\lambda\vdash n}\bleux{{(rn)^{\ell(\lambda)-1}} \binom{(r+1)\,n-2}{n-1}}\, \frac{ \,p_\lambda(\mbf{z})}{z_\lambda}  -
                            \sum_{\lambda \vdash n} \bleux{(r\, n+1)^{\ell(\lambda)-2}
                      \prod_{j\in\lambda} \binom{(r+1)\,j}{j}}\,\frac{ p_\lambda(\mbf{z})}{z_\lambda}
 \end{equation}
turns out to be $h$-positive. In terms of the corresponding $\S_n$-modules, this implies that the latter is contained in the former as  a``permutation'' $\S_n$-module. In fact, more appears to be true as discussed in Remark~\ref{remarque_quatre}. 
\end{remark}

\section{Doubly labelled polyominoes}
We may further add labels to horizontal steps of the lower path, and get an action of the product group $\S_k\times \S_n$ on the resulting set $\Labelledpolyominoes^{(2)}_{k,n}$ of doubly labelled parallelogram polyominoes.
We observe experimentally  that the associated Frobenius characteristic is given by the formula
\begin{eqnarray}
  \bleu{ {\rm Frob}(\Labelledpolyominoes_{k,n}^{(2)})(\mbf{y},\mbf{z})}&=& \bleu{\frac{1}{n-1} h_k[(n-1)\mbf{y}]\, h_n[k\,\mbf{z}]}\ +\ \bleu{\frac{1}{k-1}h_k[n\,\mbf{y}]\, h_n[(k-1)\mbf{z}]}\nonumber\\
     &&\qquad-\ \bleu{\frac{n+k-1}{(n-1)(k-1)}h_k[(n-1)\mbf{y}]\, h_n[(k-1)\mbf{z}]}.\label{double_frob}
\end{eqnarray}
Observe the evident symmetry 
\begin{equation}
  \bleux{ {\rm Frob}(\Labelledpolyominoes^{(2)}_{k,n})(\mbf{y},\mbf{z})= {\rm Frob}(\Labelledpolyominoes^{(2)}_{n,k})(\mbf{z},\mbf{y})},
\end{equation}
which corresponds to exchanging the coordinates in our polyominoes.
For example, we have
 $$ {\rm Frob}(\Labelledpolyominoes^{(2)}_{3,2})(\mbf{y},\mbf{z})= (6\,s_{{3}} ( \mbf{y} )+3\, s_{{21}} ( \mbf{y} ))\, s_{{2}} ( \mbf{z} )  +(3\,s_{{3}}
 ( \mbf{y} )  +s_{{21}} ( \mbf{y}))\, s_{{11}} ( \mbf{z} ) .$$
Here, the $s_\mu(\mbf{y})$ encode irreducibles of $\S_k$, whereas the $s_\lambda(\mbf{z})$ encode those of $\S_n$.
Naturally, our previous formula~\pref{eq:Frob} appears as the multiplicity of the trivial representation of $\S_k$ in~\pref{double_frob}. More precisely, recalling that
   $$\langle\, h_k[n\,\mbf{y}] ,h_k(\mbf{y})\,\rangle =\binom{n+k-1}{k},$$
we have
  \begin{eqnarray*}
     \bleux{ {\rm Frob}(\Labelledpolyominoes_{k,n})(\mbf{z})}&=&\bleux{ \langle\,  {\rm Frob}(\Labelledpolyominoes^{(2)}_{k,n})(\mbf{y},\mbf{z}),h_k(\mbf{y})\,\rangle}\\
               &=& \bleux{\frac{1}{n-1}\binom{n+k-2}{k}h_n[k\,\mbf{z}]}\\
                &&\qquad +\bleux{\rougex{\left({\textstyle \frac{1}{k-1}\binom{n+k-1}{k}-\frac{n+k-1}{(k-1)(n-1)}\binom{n+k-2}{k}}\right)}h_n[(k-1)\,\mbf{z}]}\\
                 &=& \bleux{\frac{1}{k}\binom{n+k-2}{k-1}h_n[k\,\mbf{z}]},
  \end{eqnarray*}
since $\left({\textstyle \frac{1}{k-1}\binom{n+k-1}{k}-\frac{n+k-1}{(k-1)(n-1)}\binom{n+k-2}{k}}\right)$ vanishes.

It also follows, directly from~\pref{double_frob}, that the total number of doubly labelled parallelogram polyominoes in $\Labelledpolyominoes^{(2)}_{k,n}$ is
\begin{equation}
   \bleu{|\Labelledpolyominoes^{(2)}_{k,n}| ={k}^{n} \left( n-1 \right) ^{k-1}+ \left( k-1 \right) ^{n-1}{n}^{k}-
 \left( k-1 \right) ^{n-1} \left( n-1\right) ^{k-1} \left( n+k-1
 \right) 
}.
\end{equation}
Considering $\S_k$-isotypic components of $\Labelledpolyominoes^{(2)}_{k,n}$, we sometimes obtain nice formulas for the Frobenius characteristic of the resulting $\S_n$-module. For instance, denoting by $\rho$ the rectangular partition $r^n$, we find that the coefficient of $s_\rho(\mbf{y})$ in $\Labelledpolyominoes^{(2)}_{rn,n}(\mbf{y},\mbf{z})$ is
\begin{equation}\label{doubly_diag}
   \bleu{\Labelledpolyominoes^{(2)}_{rn,n}(\mbf{y},\mbf{z})\Big|_{s_\rho(\mbf{y})} =\frac{1}{rn-1} h_n[(rn-1)\,\mbf{z})]}.
\end{equation}
See remark~\ref{doubly_area} for a $q$-analog of this formula. 

We also have an alternative way to doubly label parallelogram polyominoes, 
which has a meaning when such labelled objects encode stable configurations
for the sandpile model associated to the complete bipartite graph $K_{k,n}$ \cite{dukes}.
In this setting, we fix the label of the first step of the blue path to be $1$ 
(in words of the sandpile model, this step corresponds to the sink).
With this convention, we may let $\S_n\times \S_{k-1}$ act on the labels, and denote by 
${\rm Frob}(\Labelledpolyominoes^{(2,\star)}_{k,n})$ the corresponding Frobenius.

\begin{proposition}
We have:
\begin{equation}
\bleu{ {\rm Frob}(\Labelledpolyominoes^{(2,\star)}_{k,n})}
                        =\bleu{ \frac{1}{k}\,h_n[k\,\mbf{z}]\,h_{k-1}[n\,\mbf{y}]}\label{eq:Frob2star}
\end{equation}
which implies that
\begin{equation}
\bleu{ |\Labelledpolyominoes^{(2,\star)}_{k,n}|}
                        =\bleu{ k^{n-1}\,n^{k-1}},
\end{equation}
the number of parking functions (or of recurrent configurations, or of spanning trees) of $K_{k,n}$.
\end{proposition}

\begin{proof}
Take the proof of \pref{eq:Frob} and just add labels on the horizontal steps (but one) of the blue path.
\end{proof}

\section{Operators and Macdonald polynomials}
To go on with our presentation, we need to recall a few notions related to the theory of Macdonald polynomials, and operators for which they are joint eigenfunctions. The \defn{integral form Macdonald polynomials}\footnote{In previous papers, the notation $\widetilde{H}_\mu$ is often use for these, to distinguish them from other polynomials that are not used here.} ${H}_\mu(\mbf{z};q,t)$, with $\mu$ a partition of $n$, expand as 
\begin{displaymath}\label{mac_kostka}
     \bleux{{H}_\mu(\mbf{z};q,t)=\sum_{\lambda\vdash n}
       K_{\lambda,\mu}(q,t)\, s_\lambda(\mbf{z})},
   \end{displaymath}
where the $K(q,t)$'s are certain rational fractions in $q,t$, normalized so that $K_{n,\mu} =1$. 
They are entirely characterized by this normalization together with the two triangularity properties
\begin{enumerate}\itemsep=6pt
   \item[(1)]  ${{H}_\mu[\mbf{z}\,(1-q)]}$ is in the span of the $s_\lambda(\mbf{z})$, for $ \lambda\succeq \mu$ in dominance order,
   \item[(2)]  ${{H}_\mu[\mbf{z}\,(1-t)]}$ is in the span of the $s_\lambda(\mbf{z})$, for $ \lambda\succeq \mu'$ in dominance order.
\end{enumerate} 
It may be worth recalling that we have the special case
   \begin{equation}\label{H_unepart}
         \bleux{{H}_n(\mbf{z};q,t)= \prod_{i=1}^n (1-q^i)\, h_n\!\left[ \frac{ \mbf{z}}{1-q}\right]}.
   \end{equation}
 Observe that the right hand side is independent of $t$. This is helpful when we specialize the $H_\mu$'s at $t=1$, 
since in that case we have the following multiplicative property 
   \begin{equation}
         \bleux{{H}_\mu(\mbf{z};q,1)= \prod_{i=1}^{r} H_{\mu_i}(\mbf{z};q,1)},\qquad \mathrm{for}\qquad \bleux{\mu=\mu_1\mu_2\cdots\mu_r}.
   \end{equation}  
Thus, we can evaluate this last right hand side using~\pref{H_unepart}.  Observe also that this states that ${H}_\mu(\mbf{z};q,1)$ is proportional to $h_\mu[\mbf{z}/(1-q)]$. 

Another often used specialization is at $t=1/q$. In this case we have (see~\cite{remarkable})
 \begin{equation}\label{H_t_invq}
         \bleux{{H}_\mu(\mbf{z};q,1/q)= q^{-n(\mu)}\,\prod_{c\in \mu} (1-q^{{\rm hook}(c)})\, s_\mu\left[\frac{\mbf{z}}{1-q}\right]},
   \end{equation}  
with a statistic ${\rm hook}(c)$ attached to any cell $c$ in a given Ferrers diagram $\mu$.

\subsection*{Operators} For a symmetric function $f$, we consider the linear operator \bleux{$\Delta_f$}, on symmetric functions, such that
\begin{equation}\label{def_Delta}
    \bleux{\Delta_f({H}_\mu):=f[B_\mu(q,t)]\, {H}_\mu},
 \end{equation}
 where $f[B_\mu]$ stands for the evaluation of $f$ in the $q^it^j$, with $(i,j)$ varying in the set of cells of $\mu$. It is customary to denote by $\nabla$ 
the operator $\Delta_{e_n}$ restricted to symmetric functions of degree $n$. 
Observe that the operator $\Delta_{h_{k}}$ (resp. $ \Delta_{e_{k}}$) appears as the coefficient of $u^k$ in the  linear operator $\Phi$ such that
  \begin{equation}
         \bleux{\Phi({H}_\mu)=\Big(\prod_{(i,j)\in \mu} \frac{1}{1-q^it^j\,u}\Big) {H}_\mu},\qquad \mathrm{resp.}\qquad
        \bleux{\Psi({H}_\mu)=\Big(\prod_{(i,j)\in \mu} (1+q^it^j\,u)\Big) {H}_\mu}.
   \end{equation}
One of the interests of the operator $\nabla$ is that it allows a very compact formulation of the Frobenius characteristic of the $\S_n$-module of bivariate diagonal hamonic polynomials, in the form $\nabla(e_n(\mbf{w}))$ (see~\cite{haiman} for more on this).   

We will also denote $\widetilde{\Delta}_{f}$ (resp. $ \overline{\Delta}_{f}$) the
specialization at $t=1$ (resp $t=1/q$) of the operator $\Delta_f$. It turns out that the operators $\widetilde{\nabla}$, $\widetilde{\Phi}$, and $\widetilde{\Psi}$  are multiplicative, since the associated eigenvalues are multiplicative. In view of our previous observations (see~\pref{H_unepart}), the set
\begin{equation}
    \bleux{\{h_\mu[\mbf{z}/(1-q)]\ |\ \mu\vdash n\}}
 \end{equation}
 constitute a basis of eigenfunctions for the operators $\widetilde{\Delta}_{f}$. Since we are interested in calculating the effect of our operators on the elementary symmetric function $e_n$, it is worth recalling that the dual Cauchy formula implies that
 \begin{eqnarray}
    \bleux{ e_n\!\left[\mbf{z}\,\frac{1-x}{1-q}\right]} &=&\bleux{ \sum_{\mu\vdash n} s_{\mu'}[1-x]\, s_\mu\!\left[\mbf{z}\,\frac{1}{1-q}\right]}\\
      &=&\bleux{ \sum_{\mu\vdash n} f_{\mu}[1-x]\, h_\mu\!\left[\mbf{z}\,\frac{1}{1-q}\right]},
 \end{eqnarray}
where we have denoted by $f_\mu(\mbf{z})$ the \defn{forgotten} symmetric functions, that are dual to the elementary. We may specialize these expression, setting $x=q$, to expand the elementary symmetric function $e_n(\mbf{z})$ respectively in terms of the $s_\mu[\mbf{z}/(1-q)]$ and $h_\mu[\mbf{z}/(1-q)]$.

Similarly, from~\pref{H_t_invq}, it follows that the $s_\mu[\mbf{z}/(1-q)]$'s are eigenfunctions of $\overline{\Delta}_{f}$,  since they are proportional to the $H_\mu(\mbf{z};q,1/q)$'s.  The corresponding eigenvalues are $f[B_\mu(q,1/q)]$. Observe that in the case of hook shapes, {\sl i.e.} $\mu=s_{(n-r),1^r}$, these are simply
   \begin{equation}\label{hook_eigenvals}
      \bleux{f[B_{(n-r),1^r}(q,1/q)]= f\left[{q^{-r}}{\qn{n}}\right]}.
   \end{equation}

\subsection*{More operators} Recall from~\cite{HMZ} that we may introduce symmetric functions $E_{n,r}$ by means of the plethystic identity 
\begin{equation}\label{def_Enr}
     \bleux{e_n\!\left[\mbf{z}\,\frac{1-x}{1-q}\right] =\sum_{r=1}^n \frac{(x;q)_r}{(q;q)_r}\,E_{n,r}(\mbf{z})}.
\end{equation}
In particular, setting $z=q$, we find that
\begin{equation}\label{en_to_Enr}
    \bleux{e_n(\mbf{z})=\sum_{r=1}^n E_{n,r}(\mbf{z})}.
\end{equation}
These play a key role in a generalized version of the ``Shuffle Conjecture'', together with another family of linear operators on symmetric functions defined as follows.
For each integer $a$, an any symmetric function $f(\mbf{z})$, we set
 \begin{equation}\label{def_Ca}
    \bleux{C_a f(\mbf{z})= (-q)^{1-a} f\!\left[\mbf{z}-\frac{q-1}{q\,x} \right]\, \Big(\sum_{k=0}^\infty h_k(\mbf{z})\,x^k\Big)\,\Big|_{x^a} }.
\end{equation}
For example, we have
$$\begin{array}{lll}
   C_a(1)=(-q)^{1-a}\,s_a(\mbf{z}), \qquad C_a(s_1(\mbf{z}))=(-q)^{1-r}\,(s_{a,1}(\mbf{z})+q^{-1}\,s_{a+1}(\mbf{z})),\qquad \mathrm{and}\\[10pt]
   C_a(s_2(\mbf{z}))= (-q)^{1-r}\,(s_{a,2}(\mbf{z})+q^{-1}\,s_{a+1,1}(\mbf{z}))+q^{-2}\,s_{a+2}(\mbf{z})),\qquad \mathrm{if}\quad a\geq 2.
\end{array}$$
As shown in~\cite{}, the operators $C_a$'s satisfy the commutativity relations
   \begin{equation}\label{commC}
      \bleux{q\,(C_bC_a+C_{a-1}C_{b+1})= C_aC_b+C_{b+1}C_{a-1}}, \qquad {\rm whenever}\qquad \bleux{a>b}.
   \end{equation}  
Recall also that, for a compositions of $n$ into $r$ parts, we may consider the symmetric functions recursively defined as
   \begin{equation}\label{def_Ca}
    \bleux{E_\gamma(\mbf{z}):= \begin{cases}
      C_a(E_\nu(\mbf{z}))& \text{if}\ \gamma=a\nu,\\[4pt]
      1 & \text{if}\ \gamma=0,
\end{cases}}
\end{equation}
where we denote by $0$ the empty composition;
and, we have
  \begin{equation}\label{Enr_en_Egamma}
  \bleux{E_{n,r}(\mbf{z})=\sum_{\gamma\in {\rm Comp}(n,r)} E_\gamma(\mbf{z})}.
\end{equation}
\begin{remark}\label{shuffle}
An explicit enumerative description of $\nabla(E_\gamma(\mbf{z}))$ is conjectured, in~\cite[Conj. 4.5]{HMZ}, in terms of parking functions whose underlying Dyck has prescribed returns to the diagonal. It appears that a similar statement holds for $\nabla^r(E_\gamma(\mbf{z}))$, with the heights of returns to the diagonal specified by the composition $\gamma$. This is what we call the \defn{$r$-generalized shuffle conjecture}. In the particular case of $\gamma=n$, since $E_n(\mbf{z})=(-q)^{-n+1}h_n(\mbf{z})$, this conjectures implies that $\omega\,\widetilde{\nabla}^r((-q)^{-n+1}h_n(\mbf{z}))$ is the $q$-graded Frobenius characteristic of the $\S_n$-module generated by $r$-parking functions whose underlying $r$-Dyck never returns to the ``diagonal'' (of slope $1/r$), except at its end-points.
\end{remark}

Observe that the commutativity relation~\pref{commC} imply that, for any length $r$ composition $\gamma$ of $n$,  we may express  $E_\gamma(\mbf{z})$ as a linear combination 
of the $E_\mu(\mbf{z})$'s, for $\mu$ varying in the set of partitions of $n$ having $r$ parts. These last symmetric functions are in fact proportional to a specialization of Macdonald polynomials, that is 
\begin{equation}\label{Emu_en_Hmu}
   \bleux{E_\mu(\mbf{z})= q^{-n(\mu)}(-1/q)^{n-r} H_{\mu'}(\mbf{z};q,0)},
\end{equation}
where, as usual,  $n(\mu)$ stands fo $\sum_{i=1}^r (i-1)\,\mu_i$. Notice that the right hand side of this last identity is expressed in terms of the conjugate partition $\mu'$.


\section{Formula for the $q$-Frobenius of the labelled polyomino module}\label{sec:parametres}
Michele D'Adderio has proposed~\cite{dadderio} the following explicit formula for the $q$-weighted area counting version of the Frobenius characteristics  
\begin{equation}\label{michele}
   \bleu{\Frob(\Labelledpolyominoes_{k+1,n})(\mbf{z};q)= \omega\, \widetilde{\Delta}_{h_{k}}(e_n(\mbf{z}))}.
\end{equation}
Here $ \widetilde{\Delta}_{h_{k}}$ stands the specialization at $t=1$ of the operator $ {\Delta}_{h_{k}}$. This implies that 
$ \widetilde{\Delta}_{h_{k}}(e_n(\mbf{z}))$ is $e$-positive, {\sl i.e.} with coefficients in $\N[q]$ when expanded 
in the basis of elementary symmetric functions.
In generating function format, we may reformulate~\pref{michele} as 
    \begin{eqnarray*}
        \bleux{\sum_{k\geq 1} \Frob(\Labelledpolyominoes_{k,n})(\mbf{z};q)\,u^k} &=&\bleux{  \omega(u\,\widetilde{\Phi}(e_n(\mbf{z})))}\\
         &=&\omega \bleux{\sum_{\mu\vdash n}  f_\mu[1-q]  \prod_{i=1}^{\ell(\mu)} \frac{u\,h_{\mu_i}[\mbf{z}/(1-q)]}{(1-u)\cdots (1-u\,q^{\mu_i-1}) } }
    \end{eqnarray*}
Small values are as follows:
\begin{eqnarray*}
   \sum_{k\geq 1} \Frob(\Labelledpolyominoes_{k,1})(\mbf{z};q)\,u^k&=&{\frac {u}{1-u}}\,h_1(\mbf{z}),\\
     \sum_{k\geq 1}  \Frob(\Labelledpolyominoes_{k,2})(\mbf{z};q)\,u^k&=& \frac {u}{(1-u)(1-qu)} \left(h_2(\mbf{z}) +  \frac {u}{1-u}\, h_{11}(\mbf{z})\right),\\
      \sum_{k\geq 1} \Frob(\Labelledpolyominoes_{k,3})(\mbf{z};q)\,u^k&=&
    {\frac {u}{ ( 1-u)( 1-u\,q) ( 1-u\,q^2)}}\\
      &&\qquad\qquad \left( h_{{3}} (\mbf{z})+{\frac { u\,(2+q)}{1-u}}\, h_{{21}} (\mbf{z}) +{\frac { {u^2}\,( 1+q)}{ ( 1-u)^{2} }}\, h_{{111}} (\mbf{z})\right).
\end{eqnarray*}

Even if Formula~\pref{michele} is still conjectural,  the  following two formulas should help understand how $\omega{\Delta}_{h_{k}}(e_n(\mbf{z}))$ could be interpreted as a bivariate Frobenius characteristic for labelled polyominoes, with a second statistics accounting for the parameter $t$, on top of $q$ which already accounts for area. This also confirms indirectly that~\pref{michele} should hold. 
\begin{proposition}
The $q$-free component of $ \omega\, {\Delta}_{h_{k}}(e_n(\mbf{z}))$ is given by the following formula
\begin{equation}\label{prop_equation1}
  \bleu{ \omega\, {\Delta}_{h_{k}}(e_n(\mbf{z}))\big|_{q=0} = \frac{1}{t^k} \left(h_n\!\left[\mbf{z}\frac{1-t^{k+1}}{1-t}\right]-h_n\!\left[\mbf{z}\frac{1-t^{k}}{1-t}\right] \right)}.
\end{equation}
Moreover, we have
\begin{equation}\label{prop_equation2}
   \bleu{\omega\, \Delta_{h_{k-1}}(e_n(\mbf{z}))\big|_{t=1/q}= \frac{q^{n+k-nk-1}}{\qn{k}} h_n\!\left[\mbf{z}\frac{1-q^k}{1-q}\right] \qbinom{n+k-2}{k-1}}
\end{equation}
which is a direct $q$-analog of \pref{eq:Frob}.
\end{proposition}
For a combinatorial interpretation of the right hand side of~\pref{prop_equation1}, see Remark~\ref{remark_dinv_ribbon}.

\begin{proof}[\bleu{\bf Proof}.]
To prove~\pref{prop_equation1}, we first observe that 
${\Delta}_{h_{k}}(e_n(\mbf{z}))$ is symmetric in $q$ and $t$, we may thus exchange their role. We may also remove $\omega$, so that we need only show that
\begin{equation}\label{to_show}
\widehat{\Delta}_{h_k}(e_n(\mbf{z}))= \frac{1}{q^k} \left(e_n\!\left[\mbf{z}\frac{1-q^{k+1}}{1-q}\right]-e_n\!\left[\mbf{z}\frac{1-q^{k}}{1-q}\right] \right),
 \end{equation}
 where $\widehat{\Delta}_{h_k}:= \Delta_{h_{k}}\big|_{t=0}$.
Evaluating \pref{def_Enr} at $z=q^{k+1}$, we get
\begin{equation}\label{eval_Enr}
     e_n\!\left[\mbf{z}\,\frac{1-q^{k+1}}{1-q}\right] =\sum_{r=1}^n \qbinom{k+r}{r}\,E_{n,r}(\mbf{z}),
\end{equation}
since, by definition, 
   $$ \qbinom{k+r}{r}= \frac{(q^{k+1};q)_r}{(q;q)_r}.$$
Using~\pref{eval_Enr}, and well known properties of $q$-binomial coefficients, we can calculate that the right hand side of \pref{to_show} is
\begin{eqnarray}
\widehat{\Delta}_{h_k}(e_n(\mbf{z}))&=& \sum_{r=1}^n \frac{1}{q^k}\left( \qbinom{k+r}{r}- \qbinom{k+r-1}{r}\right) E_{n,r}(\mbf{z})\\
                                                                &=& \sum_{r=1}^n \ \qbinom{k+r-1}{r-1}E_{n,r}(\mbf{z}) 
\end{eqnarray}
Now, recalling \pref{en_to_Enr}, this last equation readily follows from the surprising fact that the $E_{n,r}(\mbf{z})$'s are eigenfunctions of the operator $ \widehat{\Delta}_{h_k}$. Indeed, as we will see, we have
    \begin{equation}
          \bleux{\widehat{\Delta}_{h_k}\, E_{n,r}(\mbf{z})= \qbinom{k+r-1}{r-1}E_{n,r}(\mbf{z})}. 
    \end{equation}
This  follows from the even stronger fact that for all composition $\gamma$, having $r$ parts, we have
   \begin{equation}\label{eigenfunctD}
          \bleux{\widehat{\Delta}_{h_k}\, E_{\gamma}(\mbf{z})= \qbinom{k+r-1}{r-1}E_{\gamma}(\mbf{z})}. 
    \end{equation}
In view of the observation following~\pref{Enr_en_Egamma}, we need only show that~\pref{eigenfunctD} holds for $E_\mu(\mbf{z})$, with $\mu$ a length $r$ partition. Since we have already seen recalled (see \pref{Emu_en_Hmu}) that these $E_\mu$'s are essentially specializations of the $H_\mu$'s, it only remains to check that
  $$ {\widehat{\Delta}_{h_k}\, H_{\mu'}(\mbf{z};q,0)= \qbinom{k+r-1}{r-1}H_{\mu'}(\mbf{z};q,0) }. $$
But, by definition of ${\Delta}_{h_{k}}$, this is equivalent to the classical fact that 
   $$ \qbinom{k+r-1}{r-1} =h_k\left[\frac{1-q^r}{1-q}\right]= h_k(1,q,\ldots,q^{r-1}).$$
This ends our proof of \pref{prop_equation1}.

To prove~\pref{prop_equation2}, we start by recalling from above (see~\pref{H_t_invq}) that 
the $s_\mu[\mbf{z}/(1-q)]$'s are eigenfunctions of $\overline{\Delta}_{h_{k-1}}=\Delta_{h_{k-1}}\big|_{t=1/q}$, with eigenvalue $h_{k-1}[B_\mu(q,1/q)]$.  
Once again, we remove $\omega$, to turn the identity we have to show into
   \begin{equation}\label{prop_equation2bis}
   \overline{\Delta}_{h_{k-1}}(e_n(\mbf{z}))= \frac{q^{n+k-nk-1}}{\qn{k}} e_n\!\left[\mbf{z}\frac{1-q^k}{1-q}\right] \qbinom{n+k-2}{k-1}.
\end{equation}
As we have already seen, using Cauchy's identity, we have that
 $$e_n\!\left[\mbf{z}\frac{1-q^k}{1-q}\right] =\sum_{\mu\vdash n} s_{\mu'}[1-q^k]\, s_\mu\left[\frac{\mbf{z}}{1-q}\right].$$
Since $s_\nu[1-u]$ is only non zero when $\nu$ is a hook shape, and  $s_{(n-r),1^r}[1-u]= (-u)^r(1-u)$,
we may rewrite the above identity as as
\begin{equation}\label{expand_e}
   e_n\!\left[\mbf{z}\frac{1-q^k}{1-q}\right] =\sum_{r=0}^{n-1} (1-q^k)\, (-q^k)^{n-r-1}  s_{(n-r),1^r}\left[\frac{\mbf{z}}{1-q}\right].
 \end{equation}
Using this identity at $k=1$, we may expand $e_n(\mbf{z})$ in terms of eigenfunctions for the operator $\overline{\Delta}_{h_{k-1}}$, and thus we calculate (using~\pref{hook_eigenvals})  that
   \begin{eqnarray*}
   \overline{\Delta}_{h_{k-1}}(e_n(\mbf{z}))&=&\bleux{ \overline{\Delta}_{h_{k-1}}}  \sum_{r=0}^{n-1}(1-q)\, (-q)^{n-r-1}  s_{(n-r),1^r}\left[\frac{\mbf{z}}{q-1}\right]  \\
                &=& \sum_{r=0}^{n-1}(1-q)\,  (-q)^{n-r-1}  \bleux{q^{-r(k-1)} h_{k-1}\left[\frac{1-q^n}{1-q}\right]}   s_{(n-r),1^r}\left[\frac{\mbf{z}}{q-1}\right]\\
                 &=&\bleux{\qbinom{n+k-2}{k-1}} \  \sum_{r=0}^{n-1} (-1)^{n-r-1} (1-q)\,  {q^{n-rk-1}}   s_{(n-r),1^r}\left[\frac{\mbf{z}}{q-1}\right]\\
                 &=& \bleux{\frac{q^{n+k-nk-1}}{\qn{k}}} {\qbinom{n+k-2}{k-1}}\  \sum_{r=0}^{n-1} \bleux{(1-q^k)\,  {(-q^k)^{n-r-1}}}   s_{(n-r),1^r}\left[\frac{\mbf{z}}{q-1}\right]
\end{eqnarray*}
which shows \pref{prop_equation2bis}, using \pref{expand_e} to evaluate the sum.
  \end{proof}

  \begin{remark}\label{remark_dinv_ribbon}
The right-hand side of \pref{prop_equation1} may be interpreted as a graded Frobenius 
of ribbon shaped labeled parallelogram polyominoes of size $n\times k$. 
The needed grading corresponds to the $\dinv$ 
statistic  for unlabeled parallelogram polyominoes, as defined in \cite{ADDHL}.
Recall that, in the $\AA$-word encoding $w=w_1w_2\cdots w_k$ of a parallelogram polyomino (see Figure~\ref{fig:Ppiw}), the $\dinv$ statistic is the number of pairs $(w_i,w_j)$,
with $i<j$ and $j$ is the successor of $i$ in $\AA$.
For any ribbon-shaped labeled polyomino, we may simply set its dinv-statistic to be equal to that
of the underlying  parallelogram polyomino.
With this definition, we observe that the dinv-value coincides with 
the area below the lower path of the polyomino.
This readily implies that
\begin{equation}\label{Frob-ribbon}
\Big(\Frob(\Labelledpolyominoes_{k,n})(\mbf{z};q)\Big)\Big|_{q=0}=\,
t^n\, \left(h_n\!\left[\mbf{z}\frac{1-t^{k}}{1-t}\right]-h_n\!\left[\mbf{z}\frac{1-t^{k-1}}{1-t}\right] \right)\,.
\end{equation}

\end{remark}

\begin{remark}\label{doubly_area}
Experiments suggest that we have the following area enumerating analog of Formula~\pref{doubly_diag}
\begin{equation}
   \bleu{\Labelledpolyominoes^{(2)}_{rn,n}(\mbf{y},\mbf{z};q)\Big|_{s_{\rho}(\mbf{y})} =\omega\,\widetilde{\nabla}^r((-q)^{-n+1}\,h_n)},
\end{equation}
where $\rho$ is the partition consisting of $n$ rows of size $r$. In fact, this identity follows from the $r$-generalized shuffle conjecture 
(see Remark~\ref{shuffle}), considering paths that never come back to the diagonal of slope $1/r$.
It would be interesting to describe all isotypic components of $\Labelledpolyominoes^{(2)}_{rn,n}$ in this manner.
\end{remark}

\subsection*{Toward a general dinv statistic}
A dinv-statistic\footnote{Which we don't know how to define yet.} on general labelled polyominoes would make possible an analog of the ``Shuffle Conjecture'' of \cite{HHLRU}, in the form: 
\begin{equation}
    \bleux{\Delta_{h_k}(e_n(\mbf{z}))}=\bleux{\sum_{\pi} q^{\area(\pi)}t^{\dinv(\pi)} \mbf{z}_\pi},
\end{equation}
with $\pi$ varying in the set $\Labelledpolyominoes_{k,n}$. More explicitly, this means that we would consider more general labellings of polyominos; with labels in $\N$ and allowing for repetition of labels, while keeping the strictly increasing condition alogs runs of vertical steps. In the above expression, $\mbf{z}_\pi$ denotes  the product of the variables $z_i$, for $i$ running through all labels (with multiplicities) of $\pi$.
In a sense, this should be a natural extension of the original Shuffle Conjecture, since it appears (observed experimentally) that 
\begin{equation}\label{diff_shur_pos}
    \bleux{\Delta_{h_{rn-1}}(e_n(\mbf{z})) - q^{\binom{n-1}{2}}\nabla^r(e_n(\mbf{z}))}\ \qquad {\rm lies\ in}\qquad \bleux{\N[q,t]\{s_\mu|\mu\vdash n\}.}
\end{equation}
As  we have already seen (see Remark~\ref{la_remarque}), parking functions can be considered as special cases of labelled parallelogram polymominoes, making~\pref{diff_shur_pos} a natural phenomenon.

\begin{remark}\label{remarque_quatre} Explicit computer calculations suggest that even more seems to hold. Indeed, up to the same multiplicative factor of $q^{\binom{n-1}{2}}$, the difference between $ \Delta_{h_{n-1}}(e_n(\mbf{z}))$ and the ``Shuffle-like Conjecture Formula'' for the Frobenius of the space of trivariate diagonal harmonics (see \cite[Formula~(15)]{trivariate}),  appears to always be Schur-positive (with coefficients in $\N[q,t]$). Not only is this stronger than \pref{diff_shur_pos}, but it suggest that a third statistic might play a role here.
\end{remark}

\begin{remark}
Along these lines, it is interesting to observe  that, when $k\leq n$,
\begin{equation}
   \bleux{q^{k\,n-\binom{k+1}{2}}\, \overline{\Delta}_{e_k}(e_n(\mbf{z})) = \frac{1}{\qn{k+1}}\,\qbinom{n}{k}\, e_n\!\left[\mbf{z}\,\frac{q^{k+1}-1}{q-1}\right] }.
\end{equation}
Hence, the Frobenius \pref{frobq} may be expressed as
 \begin{equation}
   \bleux{\Frob(\Labelledpolyominoes_{k,n})(\mbf{z};q)= q^{k\,n-\binom{k+1}{2}}\,\qn{k+1}\qbinom{n}{k}^{-1} \omega\,\overline{\Delta}_{e_k}(e_n(\mbf{z})) }.
\end{equation}
Observe the similarity between this expression and~\pref{michele}.
\end{remark}

\subsection*{The bounce statistic}
Angela Hicks \cite{ADDHL} has recently given a nice bijection from which it follows that we have the following explicit formula due to Garsia and D'Adderio
\begin{equation}\label{angela}
   \bleu{\PolyominoParrallelogram_{k,n}(q,t):= \sum_{\pi\in\PolyominoParrallelogram_{k,n}} q^{\area(\pi)} t^{\bounce(\pi)} = \langle \nabla(e_{k+n-2}),h_{k-1}\,h_{n-1}\rangle}
 \end{equation}
where $b(\pi)$ denotes the \defn{bounce} statistic of a polyomino $\pi$, introduced in~\cite{dukes}. In other words, up to multiplicative factor, the right hand side is the coefficient of the monomial symmetric function of $m_{k-1,n-1}$ in the expansion of $\nabla(e_{k+n-2})$. It immediately follows that $\PolyominoParrallelogram_{k,n}(q,t)$ is symmetric, both in $q,t$ and $n,k$, as already observed in~\cite{dukes}.
Moreover,  specializing $t$ at $1/q$, we get the following direct $q$-analogs of \pref{unlabeled}
    \begin{eqnarray}\label{qangela}
       \bleu{q^{(k-1)\,(n-1)}\, \PolyominoParrallelogram_{k,n}(q,1/q)}&=&\bleu{ \frac{1}{\qn{n+k}} \qbinom{n+k}{n}\qbinom{n+k-2}{n-1}}\\
          &=&\bleu{ q^{-(k-1)}\,s_{(k-1,k-1)}(1,q,q^2,\ldots,q^n)}.
    \end{eqnarray}
To prove this, we proceed as follows. Recall from \cite{remarkable} that
  $$q^{\binom{n}{2}}\nabla(e_n)\Big|_{t\leftarrow 1/q} = \frac{1}{\qn{n+1}} e_n\left[ \mbf{z}\,\qn{n+1}\right],$$
so that \pref{angela} may be reformulated as
\begin{equation} 
    {q^{\binom{n+k-2}{2}}\PolyominoParrallelogram_{k,n}(q,1/q) = \left\langle  \frac{1}{\qn{n+k-1}} e_{n+k-2}\left[ \mbf{z}\,\qn{n+k-1}\right],h_{k-1}\,h_{n-1}\right\rangle }.
\end{equation}
Now,  Cauchy formula gives
  $$e_{n+k-2}[\mbf{z}\qn{n+k-1}]= \sum_{\lambda\vdash n+k-2} s_\lambda'(\qn{n+k-1})\,s_\lambda(\mbf{z}),$$
leading us to
 \begin{equation}
    {q^{\binom{n+k-2}{2}}\PolyominoParrallelogram_{k,n}(q,1/q) =  \frac{1}{\qn{n+k-1}} e_{k-1}(\qn{n+k-1})\,e_{n-1}(\qn{n+k-1}) }.
\end{equation}
Finally, since
  $$e_j(\qn{n})=q^j\,\qbinom{n}{j},$$
 we conclude that
  \begin{equation} 
    {q^{\binom{n+k-2}{2}}\PolyominoParrallelogram_{k,n}(q,1/q) =  \frac{1}{\qn{n+k-1}}q^{k-1}\,\qbinom{n+k-1}{{k-1}}  q^{n-1}\,\qbinom{n+k-1}{{n-1}} }.
\end{equation}
A direct calculation shows that this is equivalent to~\pref{qangela}.

\section{Thanks and future considerations}
We would like to thank Vic Reiner for reminding us that we should consider doubly labelled polyominoes in the specific manner we describe here. We would also like to thank Nolan Wallach for putting us on the right track for the interpretation of parrallelogram polyominoes as indexing sets for $SL_2$-invariants.

In our view, the main question that begs to be addressed in the future, seems to find a natural context in which $\Delta_{h_k}(e_n)$ would appear as the bigraded  Frobenius characteristic of an $\S_n$-module similar to the $\S_n$-module of diagonal harmonic polynomials. This would help truly explain all our observations above.


\end{document}